\documentclass[11pt, letterpaper]{amsart}   	
\usepackage{graphicx}				
\usepackage{amssymb}

\usepackage{latexsym,exscale,enumerate,amsfonts,amssymb,mathtools}
\usepackage{amsmath,amsthm,amsfonts,amssymb,amscd,stmaryrd,textcomp}
\usepackage[normalem]{ulem}
\usepackage{young}
\usepackage{thmtools,xcolor}
\usepackage{easybmat}

\definecolor{colormy}{rgb}{0.8,0.05,0.05}
\definecolor{mycolor}{rgb}{0.25,0.99,0.25}

\usepackage[all]{xy}
\SelectTips{cm}{}

\usepackage{tikz}
\usetikzlibrary{decorations.markings}
\usetikzlibrary{decorations.pathreplacing}
\usetikzlibrary{arrows,shapes,positioning}
\tikzstyle directed=[postaction={decorate,decoration={markings,
    mark=at position #1 with {\arrow{>}}}}]
\tikzstyle rdirected=[postaction={decorate,decoration={markings,
    mark=at position #1 with {\arrow{<}}}}]

\usepackage{hyperref}

\newcommand{\Hom}{\mathrm{Hom}}

\newcommand{\Ext}{\mathrm{Ext}}

\newcommand{\Char}{\mathrm{ch}}

\newcommand{\rank}{\mathrm{rank }}
\newcommand{\coker}{\mathrm{coker}}
\newcommand{\Id}{\mathrm{Id}}
\newcommand{\Ker}{\mathrm{Ker}}
\newcommand{\im}{\mathrm{Im}}

\def\C{{\mathbb C}}

\def\Z{{\mathbb Z}}

\def\Q{{\mathbb Q}}
\def\R{{\mathbb R}}

\def\O{\mathcal O}

\theoremstyle{definition}
\newtheorem{thm}{Theorem}[section]
\newtheorem{cor}[thm]{Corollary}
\newtheorem{lem}[thm]{Lemma}
\newtheorem{prop}[thm]{Proposition}

\theoremstyle{definition}

\newtheorem{exa}[thm]{Example}

\numberwithin{equation}{section}

\declaretheorem[style=definition,name=Definition,qed=$\blacktriangle$,numberlike=thm]{defn}
\declaretheorem[style=definition,name=Remark,qed=$\blacktriangle$,numberlike=thm]{rem}

\title{BGG categories in prime characteristics}

\author{Henning Haahr Andersen}
\address{Centre for Quantum Geometry (QM), Imada,
SDU, Denmark}
\email{h.haahr.andersen@gmail.com}

\date{}							

\begin{document}

\begin{abstract}
Let $\mathfrak g$ be a simple complex Lie algebra. In this paper we study the BGG category $\mathcal O_q$ for the quantum group $U_q(\mathfrak g)$ with $q$ being a root of unity in a field $K$ of characteristic $p >0$.  We first consider the simple modules in $\mathcal O_q$ and prove a Steinberg tensor product theorem for them. This result reduces the problem of determining the corresponding irreducible characters to the same problem for a finite subset of finite dimensional simple modules. Then we investigate more closely the Verma modules in $\mathcal O_q$. Except for the special Verma module, which has highest weight $-\rho$,  they all have infinite length. Nevertheless, we show that each Verma module has a certain finite filtration  with an associated strong linkage principle. The special Verma module turns out to be both simple and projective/injective.  This leads to a family of projective modules in $\mathcal O_q$, which are also tilting modules. We prove a reciprocity law, which gives a precise relation between the corresponding family of  characters for  indecomposable tilting modules and the family of characters of simple modules with antidominant highest weights.  All these results are of particular interest when $q = 1$, and we have paid special attention to this case.

\end{abstract}

\maketitle

\section{Introduction}
The BGG-category $\mathcal O$ for a semisimple complex Lie algebra $\mathfrak g$ has been studied intensively (see e.g. \cite{Hu} and the large number of  references there) ever since it was introduced by Bernstein, Gelfand and Gelfand, \cite{BGG}. Also the corresponding version for quantum groups over $\C$ has been the subject of several papers, see e.g. \cite{AM}, \cite{EK}, \cite{Fi}, and \cite{FM}.

On the contrary rather little can be found in the literature about BGG-categories over fields of characteristic $p>0$ (the only publication in this direction which comes to mind is \cite{Fr}\footnote{S. Donkin has kindly pointed out to me in a mail of March 28, 2022 that some of the results in this paper (in particular the facts that $\Delta_q(-\rho)$ is irreducible and projective in $\mathcal O_q$ (Section 6), the results about the socles of Verma modules in Section 8.1, the existence of certain projective modules in $\mathcal O_q$, Section 8.2, and the reciprocity laws in Section 8.3) were obtained several years ago by his Ph.D-student Jonathan Dixon, see his Thesis from University of London, June 2008}). Our focus in this paper is on this case - both the modular category $\bar {\mathcal O}_p$ and the quantum case $\mathcal O_q$,  where the quantum parameter is a root of unity of odd order $\ell$ in a characteristic $p$ field. 

We prove some basic results for both these categories. Our first goal is to understand the simple modules. We prove (Section 3) a general form of Steinberg's tensor product theorem, which was first established by Steinberg, \cite{St}, in the framework of finite dimensional modular representations of a semisimple algebraic group. Then we demonstrate, how one from this result can extract the characters of all simple modules in $\bar {\mathcal O}_p$ and  $\mathcal O_q$, once the characters of all the finite dimensional simple modules are available. This last information is now determined for $\bar {\mathcal O}_p$ in terms of the socalled $p$-Kazhdan-Lusztig polynomials by the recent breakthrough in modular representation theory, \cite{AMRW}, \cite{RW} (combined with \cite{So} when $p < 2(h-1)$), whereas for $\mathcal O_q$ the same is so far only known to be the case for $p \gg 0$ (where the answer is given via the usual Kazhdan-Lusztig polynomials, cf. Proposition 3.14 below).

We then examine the Verma modules. Except for the special one - the  Verma module with highest weight $-\rho$ - they all have infinite length unlike the situation in the ordinary category $\mathcal O$ and the characteristic $0$ quantum case (see \cite{AM}). However, they all have finite $(p^r,\Delta)$-filtrations, respectively $(\ell, \Delta)$-filtrations, see Section 4,  and we show that the highest weights of the quotients in these filtrations are strongly linked to the highest weight of the Verma module in question.

We show that the special Verma module is simple (in fact the only simple Verma module), and we prove that it is also projective (Section 6). This last fact leads us to the construction of a family of indecomposable projective modules. They are projective covers of the simple modules with antidominant highest weights. We point out that this makes the simple modules in $\bar {\mathcal O}_p$ and $\mathcal O_q$ with antidominant highest weights play a special role. 

The categories $\bar {\mathcal O}_p$ and $\mathcal O_q$ contain two kinds of tilting modules. First there are the finite dimensional tilting modules which have filtrations by Weyl and dual Weyl modules, see \cite{Do} and \cite{An92}. Then there are also tilting modules of infinite dimension, namely those allowing (finite!) filtrations by Verma and dual Verma modules (Section 7).  We call these $\infty$-tilting modules, and show that for each weight $\lambda$ in the closure of the dominant chamber, there is a unique indecomposable such module with highest weight $\lambda$. We obtain (for $p \geq 2h-2$) a precise relation between the $\infty$-tilting modules and certain finite dimensional tilting modules with highest weights far inside the dominant chamber. 

We also prove (Section 8) that in fact the indecomposable $\infty$-tilting modules coincide with the projective covers of the antidominant simple modules. This leads to a reciprocity law between the number of occurrences of a Verma module in a given indecomposable $\infty$-tilting module with dominant highest weight and the multiplicity of the corresponding antidominant simple module in the $(p^r,\Delta)$-filtration, respectively the $(\ell, \Delta)$-filtration  of the Verma module. 
\vskip .3 cm 
Some 10 years ago when working with V. Mazorchuk on the characteristic $0$ quantum BGG-category resulting in the paper \cite{AW}, I promised myself one day to investigate the characteristic $p$ analogue of such categories. When P. Fiebig's recent preprint \cite{PF} resulted in some email correspondence with him, I was reminded of this promise, and this got me started on the present work.  I thank Peter for sending me his work and thus re-stimulating my interest. I am also grateful to the referee for his/her helpful comments.

\section{BGG-categories for quantum groups at roots of $1$} 

Let $K$ be a field of characteristic $p \geq 0$ and fix $q \in K^\times$. In this section we begin by recalling the definition of the quantum group $U_q$ over $K$ obtained via Lusztig's quantum divided power construction. Then we define the (integral)  $BGG$-category $\O_q$ for $U_q$ and introduce the Verma modules in $\O_q$. This leads to the standard classification of simple modules in $\O_q$. 

\subsection{The quantum group $U_q$ over $K$}
Let $\mathfrak g$ be a simple complex Lie algebra with Cartan matrix $C$ and denote by  $v$ an indeterminate. The ``generic'' quantum group $U_v = U_v(\mathfrak g) $ associated to $\mathfrak g$ is the $\Q(v)$-algebra with generators $E_i, F_i, K_i, K_i^{-1}$, $i =1, 2, \cdots , n =\rank \,  \mathfrak g$ and relations as given e.g. in \cite{Ja}, Chapter 5. This algebra has also a Hopf algebra structure,  which will carry over to the $K$-algebra $U_q$ (see \cite{APW}, Section 0)  that we now define.

Set $A = \Z[v, v^{-1}]$. Let $d \in \Z$ be non-zero. Then we have quantum numbers $[r]_d = \frac{v^{dr} - v^{-dr}}{v^d - v^{-d}} \in A$ for all $r \in \Z$. If $r>1$ we set $[r]_d!=[r]_d [r-1]_d \cdots [1]_d$.

Choose now a diagonal matrix $D$ with diagonal entries $d_1, d_2, ,\cdots , d_n \in \Z_{>0}$, which are minimal with the property that $DC$ is symmetric. Then we set 
$$ E_i^{(r)} =\frac{E_i^r}{[r]_{d_i}!} \text { and }  F_i^{(r)} =\frac{F_i^r}{[r]_{d_i}!} .$$
We now define $U_A$ to be the $A$-subalgebra of $U_v$ generated by $E_i^{(r)}, F_i^{(r)}, K_i, K_i^{-1}, \; i = 1, 2, ,\cdots , n, \; r \in \Z_{>0}$.

When $q \in K^\times$ we make $K$ into an  $A$-algebra via the homomorphism $ A \rightarrow K$ which takes $v$ to $q$. We define
$$ U_q = U_A \otimes_A K.$$
This is the quantum group (Hopf algebra) over $K$, which we shall work with. We will often abuse notation and write $E_i^{(r)}$ instead of $E_i^{(r)} \otimes 1 \in U_q$ etc. 

We have the triangular decomposition 
$$ U_q = U_q^- U_q^0 U_q^+$$ 
with $U_q^-$, respectively $U_q^+$, respectively  $U_q^0$ being the subalgebra generated by all $ F_i^{(r)}$, respectively. $E_i^{(r)}$, respectively $ K_i, K_i^{-1}, \left[\begin{smallmatrix} K_i \\ m \end{smallmatrix}\right].$ Here $\left[\begin{smallmatrix} K_i \\ m \end{smallmatrix}\right]$ is the specialization at $q$ of the element  
$\prod_{s=1}^m \frac{K_i v^{d_i(1-s)} - K_i^{-1} v^{d_i(s-1)}}{v^{d_i s} - v^{-d_i s}} \in U_A^0$.
\vskip .2 cm
We set $B_q = U_q^0 U_q^+$.

\subsection{The BGG-categories $\mathcal O_q$} \label{BGG cat}

Set $X = \Z^n$ and let $R \subset X \otimes_\Z \R$ be the root system associated with $C$. We choose a set of simple roots $S = \{\alpha_1, \alpha_2, \cdots , \alpha_n\}$ in $R$ and denote by $R^+$ the corresponding set of positive roots. The ordering on $X$ induced by $R^+$ is denoted $\leq$. By $N$ we denote the number of positive roots. The generators $E_i $ and $F_i$ above correspond to $\alpha_i$, resp. $-\alpha_i$. 

We call the elements of $X$ weights and define as usual the set of dominant weights
$$ X^+ = \{\lambda \in X \mid \langle\lambda, \alpha_i^\vee \rangle \in \Z_{\geq 0} \text { for all } i = 1, 2, \cdots , n \}.$$
Here $\beta^\vee$ is the dual root of $\beta \in R$. We let $R^\vee$ denote the dual root system.

Let $\lambda \in X$ and set $\lambda_i = \langle \lambda, \alpha_i^\vee \rangle$ . Then $ \lambda$ defines a character
$$ \chi_\lambda : U^0_q \rightarrow K,$$
which takes $K_i$ to $q^{d_i \lambda_i}$ and $\left[\begin{smallmatrix} K_i \\ m \end{smallmatrix}\right]$ to $\binom {\lambda_i}{m}$, \cite{APW}, Lemma 1.1. We extend $\chi_\lambda$ to $B_q$ by setting $\chi_\lambda (E_i^{(r)})  = 0$ for all $i= 1, 2, \cdots , n, \; r \in \Z_{>0}$.

For any $U_q^0$-module $M$ we define the $\lambda$-weight space in $M$ to be 
$$ M_ \lambda = \{m\in M \mid u m = \chi_\lambda (u) m, \; u \in U^0_q \}.$$
We say that $M$ is a weight module if 
$$ M = \bigoplus_{\lambda \in X} M_\lambda,$$
and we call $\lambda$ a weight of $M$ if $M_\lambda \neq 0$.
Now we are ready to define the $BGG$-category we are going to study.

\begin{defn}
The category $\mathcal O_q$ is the full subcategory of the category of $U_q$-modules consisting of those $M$, which satisfy
\begin{enumerate}
\item $M$ is a weight module,
\item $\dim_k (M_\lambda) < \infty$ for all $\lambda$, and there exist finitely many $\lambda_1, \lambda_2, \cdots , \lambda_r \in X$ such that for any weight $\lambda$ of $M$ we have $\lambda \leq \lambda_i$ for some $i \in \{1, 2, \cdots ,r\}$,
\item for all $m \in M$ we have $\dim_K(B_q m) < \infty.$
\end{enumerate}
\end{defn}

\begin{rem}
In \cite{Hu} (as well as in \cite{AM}) the $BGG$-category is defined to consist of finitely generated $U_K$ (or $U_q$) -modules. We have chosen the weaker finiteness condition (2), because it ensures that $\mathcal O_q$ is an abelian category for all $p$. Moreover, it also implies that the duality functor (as defined as in \cite{Hu}, Section 3.3) preserves $\mathcal O_q$. 
\end{rem}

Note that $\mathcal O_q$ consists of modules with integral weights. Moreover, our definition of weight modules implies that all modules in $\mathcal O_q$ have type $\bf 1$. 

When $M \in \mathcal O_q$, the character  $\Char M$ is given by
$$\Char M = \sum_{\lambda \in X} \dim_K(M_\lambda) e^\lambda.$$

\subsection{Verma modules and simple modules in $\mathcal O_q$} \label{Verma and simple}

Let $\lambda \in X$. Exactly as in the classical case we define the Verma module $\Delta_q(\lambda) \in \mathcal O_q$ by
$$ \Delta_q(\lambda) = U_q \otimes_{B_q} \lambda,$$
where $\lambda$ denotes the $1$-dimensional $K$-space with $B_q$-action given by $\chi_\lambda $.

Note that $\Delta_q(\lambda)_\mu \neq 0$ iff $\mu \leq \lambda$. Moreover, $\dim_K \Delta_q(\lambda)_\lambda= 1$.

As $U_q$ has a $PBW$-bases we see that the character of $\Delta_q(\lambda)$ is given by
\begin{equation} \label{Verma char}
 \Char \Delta_q(\lambda) = \sum_{\mu \leq \lambda} P(\lambda - \mu) e^\mu.
 \end{equation}
Here $P$ is the Kostant partition function. 

We set $q^- = \Char \Delta_q(-\rho) =\sum_{\mu \leq -\rho} P(-\rho - \mu) e^\mu $. Then by the above formula we get  for arbitrary $\lambda \in X$ 
$$ \Char \Delta_q(\lambda) = q^-  e^{\lambda + \rho}.$$

\begin{rem} Here we are using a notation similar to the one used in \cite{Hu}. In fact, our $q^-$ is the inverse of Humphreys $q$-function (see \cite{Hu}, Section 2.3). Note that the notation $q^-$ has nothing to do with the quantum parameter $q$.
\end{rem}

Clearly $\Delta_q(\lambda)$ has a unique maximal proper submodule, namely the sum of all submodules $M \subset \Delta_q(\lambda)$ with $M_\lambda = 0$. In other words, $\Delta_q(\lambda)$ has a unique simple quotient. We shall denote this quotient $L_q(\lambda)$. 

\begin{thm}
\begin{enumerate}
\item The set $\{L_q(\lambda) \mid \lambda \in X\}$ is up to isomorphisms a complete list of simple modules in $\mathcal O_q$.
\item $L_q(\lambda)$ is finite dimensional iff $\lambda \in X^+$.
\item Let $\lambda \in X^+$. If $q^m \neq 1$ for all $m>0$ then $\Char L_q(\lambda) = \chi (\lambda)$, where $\chi (\lambda)$ is the Weyl character.
\end{enumerate}
\end{thm}

\begin{proof} (1) and (2) are proved exactly as in the classical case, \cite{Hu}, Sections 1.3 and 1.6. The proof of (3) is given in \cite{APW}, Theorem 6.4.
\end{proof}

\begin{rem} When $p = 0$ the equality in (3) also holds for $q = \pm 1$, see \cite{APW}, Remark following Theorem 6.4.
\end{rem}

\section{Steinberg's tensor product theorem in $\mathcal O_q$}

We continue to denote by $K$ an arbitrary field. In this section $q \in K$ will be a root of unity of order $\ell$. For convenience we shall assume $\ell $ is odd (to extend to even $\ell$ see \cite{An96} and \cite{An03})  and not divisible by $3$ when $R$ is of type $G_2$. When $q = 1$ and $p = 0$ the category $\mathcal O_q$ is the ordinary category $\mathcal O$. The case $\ell > 1$ and $p = 0$ was treated in \cite{AW}. So in this paper we focus on $p > 0$. When $\ell = 1$  we may identify $\mathcal O_q$ with the modular category $\bar {\mathcal O}_p$ for the hyperalgebra of $\mathfrak g$. In the following we will always assume $\ell >1$ (i.e. $q$ a non-trivial root of $1$), using special notation (bars and subscript $p$) when we are in the modular category, i.e. when $q = 1$.

\subsection{The small quantum group} \label{small q}
The small quantum group $u_q$ is the subalgebra of $U_q$ generated by $E_i, F_i, K_i, \; i=1, 2, ,\cdots , n$, cf. \cite{Lu}. Just like for $U_q$ we have a triangular decomposition  $u_q = u_q^- u_q^0 u_q^+$. As $E_i^\ell = [\ell]_{d_i}! E_i^{(\ell)} = 0$ (because by our assumption $[\ell]_{d_i} = 0$) we see that $u_q^+$ is finite dimensional. In fact, via the PBW-bases we obtain $\dim u_q^+ = \ell^N$, where $N$ is the number of positive roots, cf. the analogous modular case treated in \cite{RAG}, II.3.3. Likewise, $\dim u_q^- = \ell^N$. As $K_i^{2\ell} = 1$ we have also that  $ u_q^0 $ (and hence $u_q$) is finite dimensional. 

It will be convenient for us instead of $u_q$-modules to consider finite dimensional modules for the subalgebra $u_qB_q$ of $U_q$. Note that $u_qB_q$ has the triangular decomposition $u_qB_q \simeq u_q^- \otimes U_q^0 \otimes U_q^+$.

\subsection{Baby Verma modules and simple $u_qB_q$-modules}

Let $\lambda \in X$. Like in Section 2 we consider $\lambda$ as a $1$-dimensional $B_q$-module. 

The baby Verma module corresponding to $\lambda$ is defined by 
$$ \tilde \Delta_\ell(\lambda) = u_qB_q \otimes_{B_q} \lambda.$$
By the facts recalled in Section \ref{small q} we have $\dim \tilde \Delta_\ell(\lambda) = \ell^N $. Moreover, the fact that $u_qB_q = u_q^- \otimes B_q$ implies that weight multiplicities in baby Verma modules are given by 
\begin{equation} \label{baby char}
\dim_K \tilde \Delta_\ell(\lambda)_{\lambda - \mu} = \# \{ (n_\beta)_{\beta \in R^+} \mid 0 \leq n_\beta < \ell \text { for all } \beta \text { and } \sum_{\beta \in R^+} n_\beta \beta = \mu \}
\end{equation}
This formula is identical to its modular analogue in \cite{RAG}, Section II.9.2.

Note that for any $\mu \in X$ the character $\ell \mu$ of $B_q$ extends to a character of $u_qB_q$. In fact, $\ell \mu$ restricts to the trivial character on $u_q^0$. So $\ell \mu$ extends from $B_q$ to $u_qB_q$ by defining it to be trivial on $u_q$.

The same arguments as in Section 2.3 show that $\tilde\Delta_\ell (\lambda)$ has a unique simple quotient, and if we denote this $\tilde L_\ell (\lambda)$, then the set of simple 
modules $(\tilde L_\ell(\lambda))_{\lambda \in X}$ is up to isomorphism a full set of non-isomorphic  finite dimensional $u_qB_q$-modules (of type $\bf 1$).

By the tensor identity (or just from the definition of  $\tilde \Delta_\ell$) combined with the above observation about $\ell$-multiples of characters in $X$ we get 

\begin{equation} \label{period} \tilde \Delta_\ell(\lambda + \ell \mu) \simeq \tilde \Delta_\ell(\lambda) \otimes \ell \mu \text { and } \tilde  L_\ell(\lambda + \ell \mu) =\tilde  L_\ell(\lambda) \otimes \ell \mu \text { for all } \lambda, \mu \in X.
\end{equation}

Set now
$$X_\ell = \{ \lambda \in X^+ \mid \langle \lambda, \alpha_i^\vee\rangle < \ell \text { for all } i= 1, 2, \cdots, n\}.$$
This is the set of $\ell$-restricted weights. An important result (the quantum analogue of Curtis' theorem, \cite{Cu},  for the algebraic group over $K$ corresponding to $\mathfrak g$) says that the simple $u_q$-modules are in fact restrictions of the simple $U_q$-modules in $\mathcal O_q$ with $\ell$-restricted highest weights (see \cite{AW}, Theorem 1.9). 
\begin{thm}\label{Curtis_q} If $\lambda \in X_\ell$ then $L_q(\lambda)$ remains simple when restricted to $u_q$, i.e. we have $L_q(\lambda)_{\mid_{u_q}} \simeq L_\ell(\lambda)$.
\end{thm}

An immediate consequence of this theorem and the second part of (\ref{period}) is
\begin{cor}
Let $\lambda \in X$ and write $\lambda = \lambda^0 + \ell \lambda^1$ for (unique) $ \lambda^0 \in X_\ell$ and $\lambda^1 \in X$. Then 
$$ \tilde L_\ell(\lambda) \simeq L_q(\lambda^0)_{\mid_{u_qB_q}} \otimes \ell \lambda^1.$$
\end{cor}

\subsection{The Steinberg module for $U_q$}

The Steinberg module for $U_q$ is
$$ St_q = L_q((\ell - 1)\rho),$$
where $\rho$ as usual denotes half the sum of the positive roots in $R$. By the strong linkage principle for quantum groups at roots of $1$, \cite{An03}, Corollary 4.5, we have also $St_q = H^0_q((\ell -1)\rho)$. Here  $H^0_q((\ell -1)\rho)$ is the dual Weyl module with highest weight $(\ell -1)\rho$. Its character is $\chi((\ell -1)\rho)$, and it has dimension $\ell^N$.  We conclude that when we restrict to $u_qB_q$ we have 
$$ St_q = \tilde \Delta_q((\ell -1)\rho) = \tilde L_q((\ell -1)\rho).$$

Later we shall need the following result.

\begin{prop} \label{irr baby Verma} Let $\lambda \in X_\ell$. Then 
$$ L_\ell(\lambda) = \Delta_\ell(\lambda) \text { iff } \lambda = (l-1)\rho.$$
\end{prop}
\begin{proof} By Theorem \ref{Curtis_q} we have $L_\ell(\lambda) = L_q(\lambda)_{\mid_{u_q}}$. Now $L_q(\lambda)$ is a submodule of the dual Weyl module $H^0_q(\lambda)$. Therefore, $\dim_K L_\ell(\lambda) \leq \dim_K H^0_q(\lambda)$. But the latter dimension is given by Weyl's dimension formula and it is immediate that this is strictly less than $\ell^N$ if $\lambda \in X_\ell \setminus \{(\ell -1)\rho\}$.
\end{proof}

\subsection{The quantum Frobenius homomorphism}

Let $\bar U = U_\Q(\mathfrak g)$ be the ordinary enveloping algebra of $\mathfrak g$ over $\Q$. We choose the standard Chevalley generators $e_i, f_i, h_i$, $i= 1, 2, \cdots , n$ for $\mathfrak g$. Via the corresponding Kostant $\Z$-form $\bar U_\Z$ in $\bar U$ we obtain the hyperalgebra $\bar U_K = \bar U_\Z \otimes_\Z K$ for $\mathfrak g$ over $K$. 
There is a quantum Frobenius homomorphism (see \cite{Lu}, Section 8,  \cite{AW}, 1.2)
$$ F_q: U_q \rightarrow \bar U_K,$$
which is determined by 
$$ F_q(E_i^{(r)}) = \begin{cases} {e_i^{(\frac{r}{\ell})} \text { if } \ell \mid r,}\\ {0 \text { otherwise,}} \end{cases}$$
$$F_q(F_i^{(r)}) = \begin{cases} {f_i^{(\frac{r}{\ell})} \text { if } \ell \mid r,}\\ {0 \text { otherwise,}} \end{cases} $$
$$ F_q(K_i) = 1, $$
and
$$ F_q(\left[\begin{smallmatrix} K_i \\ m \end{smallmatrix}\right]) = \begin{cases} {\binom {h_i} {\frac{m}{\ell}} \text { if } \ell \mid m,} \\ {0 \text { otherwise. }} \end{cases} $$

This allows us to make any $\bar U_K$-module $M$ into a $U_q$-module. We denote the $U_q$-module obtained in this way by $M^{[\ell]}$. Note that $u_q$ acts trivially on 
$M^{[\ell]}$. On the other hand, if $L$ is a $U_q$-module which is trivial when restricted to $u_q$, then there exists a $\bar U_K$-module $\bar L$ such that $L = \bar L^{[\ell]}$. In this case we write $\bar L = L^{[-\ell]}$.

\subsection{The quantum Steinberg tensor product theorem in $\mathcal O_q$} \label{SecSTP_q}

We define the $BGG$-category $\bar {\mathcal O}_p$ for $\bar U_K$ just like we defined $\mathcal O_q$ for $U_q$ in Section \ref{BGG cat}. The Verma modules in $ \bar {\mathcal O}_p$ are denoted $\bar \Delta_p(\lambda), \; \lambda \in X$, and the simple quotient of $\bar \Delta_p(\lambda)$ is $\bar L_p(\lambda)$. When $p = 0$ we are in the ordinary integral category situation, i.e. we have $\bar {\mathcal O}_0 = \mathcal O_{int}$, see \cite{AM}.

The quantum Steinberg tensor product theorem is now the following result.

\begin{thm} \label{STP_q} Let $\lambda \in X$ and write $\lambda = \lambda^0 + \ell \lambda^1$ with $\lambda^0 \in X_\ell$. Then we have an isomorphism in $\mathcal O_q$
$$ L_q(\lambda) \simeq L_q(\lambda^0) \otimes \bar L_p(\lambda^1)^{[\ell]}.$$
\end{thm}

\begin{proof} When $p = 0$ this was proved in \cite{Lu89}, Theorem 7.4. We gave a  different proof in \cite{AM}, Theorem 3.1 (also in the case $p = 0$). This proof carries over to arbitrary $p$.
\end{proof}

\begin{rem} The original Steinberg tensor product theorem, \cite{St} dealt with finite dimensional simple modules for seminsimple algebraic groups in characteristic $p>0$. The most general quantum version for finite dimensional simple modules known to the author is \cite{AW}, Theorem 1.10.
\end{rem}

\subsection{Steinberg's tensor product theorem in $\bar {\mathcal O}_p$}\label{SecSTP_p}
In this section we assume $p > 0$.

Consider the ``usual'' Frobenius homomorphism $\bar F:\bar  U_K \rightarrow \bar U_K$. It is given on generators by
$$ e_i^{(r)} \mapsto \begin{cases} {e_i^{(\frac{r}{p})}\text { if } p \mid r,} \\ { 0 \text { otherwise,}} \end{cases}$$
$$ f_i^{(r)} \mapsto \begin{cases} {f_i^{(\frac{r}{p})} \text { if } p \mid r, } \\ { 0 \text { otherwise,}} \end{cases}$$
$$ \binom{h_i}{m} \mapsto \begin{cases} { \binom{h_i}{\frac{m}{p}} \text { if } p \mid m}, \\ {0 \text { otherwise.}} \end{cases}$$

Twisting a $\bar U_K$-module $M$ by $\bar F$ we obtain a new $\bar U_K$-module which we denote $M^{(1)}$. We can iterate the twisting  and get in this way modules $M^{(r)}$,  $r \geq 0$. We have corresponding {\it small hyperalgebras} (or $r$-restricted enveloping algebras) $\bar u_r$, namely the subalgebras generated by all the $e_i^{(m)}$ and $f_i^{(m)}$ with $m < p^r$.

 The same arguments as in Section \ref{SecSTP_q} lead to the following result.
\begin{thm} \label{STP_p} Let $\lambda \in X$ and write $\lambda = \lambda^0 + p \lambda^1, \; \lambda^0 \in X_p, \; \lambda^1 \in X$. Then we have an isomorphism in $\bar {\mathcal O}_p$
$$ \bar L_p(\lambda) \simeq \bar L_p(\lambda^0) \otimes \bar L_p(\lambda^1)^{(1)}.$$
\end{thm}

\begin{rem} In this theorem we have abused notation and used the same notation for the $p$-adic expression for $\lambda$ as we have earlier used for the $\ell$-adic expression. We will continue to do this in the following, sometimes also replacing $p$ by some higher $p$-power.  It should be clear from the context which case we work with.
\end{rem}

Iterating Theorem \ref{STP_p} $r$-times we get
\begin{cor} \label{iterate} Suppose $\lambda \in X$. Write $\lambda = \lambda^0 + p^r\lambda^1$ with $\lambda^0 \in X_{p^r}$ and $\lambda^1 \in X$. Then in $\bar {\mathcal O}_p$ we have
$$ \bar L_p(\lambda) \simeq \bar L_p(\lambda^0) \otimes \bar L_p(\lambda^1)^{(r)}.$$
\end{cor}

\begin{rem} Consider $\lambda = (p^r-1)\rho, \, r>0$. In this case we use the notation $\overline{St}_r = L((p^r -1)\rho)$. This is the $r$-th Steinberg module in $\bar {\mathcal O}_p$. Applying Theorem \ref{STP_p} several times we get $\overline{St}_r = \overline{St}_1 \otimes \overline{St}_1^{(2)} \otimes \cdots \otimes \overline{St}_1^{(r-1)}$.
\end{rem}

\subsection{Applications of the Steinberg theorems}
In this section we derive several consequences of Theorems \ref{STP_q} and \ref{STP_p}. First we apply Corollary \ref{iterate} to obtain

\begin{cor}\label{first appl} Let $\lambda \in X$ and $r \in \Z_{>0}$. Write $\lambda = \lambda^0 + p^r \lambda^1$ with $\lambda^0 \in X_{p^r}$, $\lambda^1 \in X$. Fix $\mu \leq \lambda$. Then we have
$$ \dim_K \bar L_p(\lambda)_\mu = \dim_K \bar L_p(\lambda^0)_{\mu - p^r\lambda^1}$$
for all $r \gg 0$.
\end{cor}

\begin{proof}
By Corollary \ref{iterate} we have
\begin{equation} \label{weights of simple}
\dim_K \bar L_p(\lambda)_\mu = \sum_{\nu, \eta} \dim_K \bar L_p(\lambda^0)_\nu \dim_K \bar L_p(\lambda^1)_\eta,
\end{equation}
where the sum runs over all $ \nu \leq \lambda^0 $ and $\eta \leq \lambda^1$ satisfying $\nu + p^r \eta = \mu$. We write $ \nu = \lambda^0 - \sum_{\alpha \in S} m_\alpha \alpha$, $\eta = \lambda^1 - \sum_{\alpha \in S} k_\alpha \alpha$ and $\mu = \lambda - \sum_{\alpha \in S} n_\alpha \alpha$ with $m_\alpha, k_\alpha, n_\alpha \in \Z_{\geq 0}$. The condition $\nu + p^r \eta = \mu$ then means
$$ m_\alpha + p^r k_\alpha = n_\alpha \text { for all } \alpha \in S.$$
Hence if  $p^r > n_\alpha$ for all $\alpha \in S$ we must have $k_\alpha = 0$ and $m_\alpha = n_\alpha$ for all $\alpha \in S$. In other words, the equation (\refeq{weights of simple}) contains just $1$ summand on the right hand side, namely the one with  $\eta = \lambda^1$ and $\nu = \lambda^0  -(\lambda - \mu) = \mu - p^r \lambda^1$.  This gives the stated formula.

\end{proof}
\begin{rem} 
\begin{enumerate}
\item Note that the proof gives an explicit bound for how large $r$ needs to be in order for the formula in this corollary to hold for a given $\mu = \lambda - \sum_{\alpha \in S} n_\alpha \alpha$, namely  $r > n_\alpha$ for all $\alpha \in S$.
\item The same formula was obtained by P. Fiebig, see \cite{PF},  Theorem 1,  with a different proof.
\end{enumerate}
\end{rem}
\begin{cor}\label{irr in Oq}
The characters $\{\Char L_q(\lambda) \mid \lambda \in X\}$ of the irreducible modules in $\mathcal O_q$ are determined by the two finite subsets $\{\Char L_q(\lambda) \mid  \lambda \in X_\ell\}$ and $\{\Char \bar L_p(\lambda) \mid \lambda \in X_p\}$.
\end{cor}

\begin{proof} Let $\lambda \in X$ and write $\lambda = \lambda^0 + \ell \lambda^1$ with $\lambda^0 \in X_\ell$ and $\lambda^1
 \in X$. By Theorem \ref{STP_q} we have $\Char L_q(\lambda) =\Char L_q(\lambda^0
 ) (\Char \bar L_p(\lambda^1))^{(\ell)}$ Here we use the notation  $(\sum_\mu a_\mu e^\mu)^{(\ell)} = \sum_\mu a_\mu e^{\ell \mu}$. Now apply Corollary \ref{first appl}.
\end{proof}

\begin{rem} The characteristic $0$ analogue of this corollary says that in that case the irreducible characters in $\mathcal O_q$ are given by the finite set $\{ch L_q(\lambda) \mid \lambda \in X_\ell\}$ together with the irreducible characters in ordinary (integral) category $\mathcal O$, cf. \cite{AM}, Section 5. 
 The latter characters are given by the Kazhdan-Lusztig polynomials for the Weyl Group $W$ of $\mathfrak g$, \cite{KL},  whereas the former set of characters are determined by the Kazhdan-Lusztig polynomials for the affine Weyl group $W \ltimes \Z R^\vee$, see \cite {KT1} and \cite{KT2}. 
 In characteristic $p$ the situation is more complicated: the modular irreducible characters $\{\Char \bar L_p(\lambda)\mid \lambda \in X_p\}$ are determined by the $p$-Kazhdan-Lusztig polynomials (at least for $p \geq 2h-1$), see \cite {AMRW}, \cite{RW}. 
 The quantum irreducible characters $\{\Char L_q(\lambda) \mid \lambda \in X_\ell\}$ are (at least for $p \gg 0$) identical to their characteristic $0$ counterparts, as we shall now prove.
\end{rem}

\begin{prop} Let $\zeta \in \C$ be a primitive $\ell$-th root of $1$. Then for $p \gg 0$ we have
$$ \Char L_q(\lambda) = \Char L_\zeta(\lambda)$$
for all $\lambda \in X_\ell$.
\end{prop}

This result was noticed many years ago by L. Thams in his PhD-thesis (Aarhus University, 1993). As this thesis is not easily available we have reproduced the proof below.

\begin{proof}
Let $\lambda \in X^+$. We shall use notation from \cite{APW}. For instance, $H^N_q(w_0 \cdot \lambda)$ will denote the Weyl module for $U_q$ with highest weight $\lambda \in X^+$, and $H^0_q(\lambda)$ is the corresponding dual Weyl module (the restrictions on $\ell$ in \cite{APW} are no longer needed as the quantum Kempf's vanishing theorem have been proved in general, see \cite{SR-H}).

The irreducible module $L_q(\lambda)$ is the image of the natural homomorphism $c_q: H^N_q(w_0\cdot \lambda) \rightarrow H^0_q(\lambda)$. Both these modules have $A$-forms, which we denote $H^N_A(w_0\cdot \lambda)$ and $H^0_A (\lambda)$, respectively. Also $c_q$ lifts to a $U_A$-homomorphism $c_A: H^N_A(w\cdot\lambda) \rightarrow H_A^0(\lambda)$. This is an injection because the corresponding homomorphism over the fraction field $\Q(v)$ is an isomorphism, cf \cite{APW}, Theorem 6.4. Setting $C_A(\lambda) = \coker(c_A)$ we have the short exact sequence 
$$ 0 \rightarrow H_A^N(w_0\cdot \lambda) \rightarrow H^0_A(\lambda) \rightarrow C_A(\lambda) \rightarrow 0.$$
Tensoring this sequence by $A_\ell = A/(\Phi_\ell))$, where $\Phi_\ell$ is the $\ell$-th cyclotomic polynomial in $A$ we get the exact sequence 
\begin{equation} \label{ses A_l}
 H_{A_\ell}^N(w_0\cdot \lambda) \rightarrow H^0_{A_\ell}(\lambda) \rightarrow C_A(\lambda)\otimes_A A_\ell \rightarrow 0.
 \end{equation}
Recall that we consider $K$, respectively $\C$ as $A$-algebras via the homomorphisms taking $v \in A$ into $q$ and $\zeta$, respectively. Note that both these structure maps factor through $A_\ell$. So when we tensor (\refeq{ses A_l}) by $K$, respectively $\C$ we get exact sequences
$$ H^N_q(w_0\cdot \lambda) \rightarrow H^0_q(\lambda) \rightarrow C_{A_\ell}(\lambda) \otimes_{A_\ell}  K\rightarrow 0, $$
respectively
$$ H^N_\zeta (w_0\cdot \lambda) \rightarrow H^0_\zeta (\lambda) \rightarrow C_{A_\ell}(\lambda) \otimes_{A_\ell} \C \rightarrow 0 $$
These two sequences show that $\Char L_q(\lambda) = \Char L_\zeta (\lambda)$ if and only if  $\dim _K  (C_{A_\ell}(\lambda) \otimes_{A_\ell}  K) = \dim_\C (C_{A_\ell}(\lambda) \otimes_{A_\ell} \C)$. This holds in turn iff $C_{A_\ell}(\lambda)$ has no $p$-torsion. As $C_{A_\ell}(\lambda)$ is finitely generated, this will certainly be true if  $p$ is big enough, say $p > m(\lambda)$ where $m(\lambda) \in \Z_{>0}$. We conclude that the proposition holds for $p > M(\ell) = \max\{m(\lambda) \mid \lambda \in X_\ell \}$.

\end{proof}
\begin{rem} I have no estimate for the $M(\ell)$ occurring in this proof.
\end{rem}

\subsection{Frobenius homomorphisms and $\Hom$-space identities}

For later use we record here a couple of general identities for $\Hom$-spaces related to the quantum and the modular  Frobenius homomorphisms from Section \ref{SecSTP_q}, respectively Section \ref{SecSTP_p}.

\begin{prop} \label{homid} \begin{enumerate}
\item Let $M$ and $N$ be two $U_q$-modules and $L$ be a $\bar U_K$-module. Then the quantum Frobenius homomorphism $F_q$ gives an isomorphism
$$ \Hom_{U_q}(M \otimes L^{[\ell]}, N) \simeq \Hom_{\bar U_K}(L, \Hom_{u_q}(M, N)^{[-\ell]}).$$
\item Let $L, M$ and $N$ be three $\bar U_K$-modules. Then for each $r > 0$ we have an isomorphism
$$\Hom_{\bar U_K}(M \otimes L^{(r)}, N) \simeq \Hom_{\bar U_K}(L, \Hom_{\bar u_r}(M, N)^{(-r)}).$$
\end{enumerate}
\end{prop}
\vskip .5 cm 
The identities in this proposition are the Hopf-algebra versions of the corresponding ones related to  a group homomorphism $f: G  \rightarrow G'$. If $H$ denotes the kernel of $f$ then for any $G$-module $M$ we have an obvious identification $M^G = (M^H)^{G/H}$. On $\Hom$-spaces this leads to identities
$$ \Hom_G(M \otimes L, N) = \Hom_{G/H}(L, \Hom_H(M, N))$$
for any two  $G$-modules $M$ and $N$, and any $G/H$-module $L$.

\section{Linkage principles}
In this section we shall derive linkage principles for Verma modules, first in $\bar {\mathcal O_p}$ and then in $\mathcal O_q$. To state them we shall need the affine Weyl groups $W_p$ and $W_\ell$.  These are the groups generated by the affine reflections $s_{\beta, m}$, where $\beta  \in R^+$ and $m\in \Z$, defined by
$$ s_{\beta, n} \cdot \lambda = s_\beta \cdot \lambda + mp \beta, \; \lambda \in X$$
in the case of $W_p$. For $W_\ell$ we replace $p$ by $\ell$ in this formula. 

\subsection{Linkage principles for Verma modules in $\bar {\mathcal O_p}$}
We assume $p > 0$ and use the notation from Section 3.

 The hyperalgebra $\bar U_K$ has triangular decomposition $\bar U_K = \bar U_K^{-} \bar U_K^{0} \bar U_K^+$, where $\bar U_K^-$, respectively $\bar U_K^+$,  is the subalgebra generated by all $f_i^{(r)}$, respectively $e_i^{(r)}$, and $\bar U_K^0$ is generated by all $\binom{h_i}{m}$. We set $\bar B_K = \bar U_K^0 \bar U_K^+$. Then the Verma module in $\bar {\mathcal O_p}$ with highest weight $\lambda \in X$ is 
$$\bar \Delta_p(\lambda) = \bar U_K \otimes _{\bar B_K} \lambda.$$

Recall from Section 3.6 that we have the small hyperalgebra $\bar u_1$ is of $\bar U_K$  generated by $f_i, h_i, e_i \; i= 1, 2, 3, ,\cdots , n$., and corresponding higher versions $u_r$ for all $r > 1$. We have then baby Verma modules defined  by
$$ \tilde \Delta_r(\lambda) = \bar u_r \bar B_K \otimes_{\bar B_K} \lambda, \; \lambda \in X.$$
These modules have simple quotients denoted $\tilde  L_{p^r}(\lambda)$.

Recall from \cite{RAG}, II. 9.11 (or see \cite{Doty} for a slightly stronger version) that we have the following linkage principle for baby Verma modules.

\begin{thm} \label{SLP_r} Let $\lambda, \mu \in X$ and $r \geq 1$. If $\tilde L_{p^r}(\mu)$ is a composition factor of $\tilde \Delta_{r}(\lambda)$ then $\mu \uparrow \lambda$.
\end{thm}
Here $\uparrow$ is the strong linkage relation, see \cite{RAG}, II.6.4:
\begin{defn} \label{defSL}
Let $\lambda, \mu \in X$. We say that $\mu$ is strongly linked to $\lambda$ and write $\mu \uparrow \lambda$ iff there exist weights $ \mu_1, \mu_2, \cdots , \mu_r \in X$ and reflections $s_1, s_2, \cdots , s_{r-1} \in W_p$ such that 
$$\mu_{i+1} = s_i \cdot \mu_i \text { and }
\mu = \mu_1 \leq \mu_2 \leq  \mu_3 \leq \cdots \leq \mu_r = \lambda.$$
\end{defn}
Let now $M$ be an arbitrary finite dimensional $\bar u_{r} \bar B_K$-module (as always we assume $M$ is also a weight module, i.e. $M$ is of type $\bf 1$). Then we set 
$$ \bar \Delta'_{p^r} (M) = \bar U_K \otimes_{\bar u_{r}\bar B_K} M.$$
Clearly $\bar \Delta'_{p^r}$ is an exact functor from the category of finite dimensional  $\bar u_{r} \bar B_K$-modules into $\bar {\mathcal O}_p$. In addition it has the following properties.

If $V$ is a $\bar B_K$-module, then the Frobenius twist $V^{(r)}$ is a $\bar u_{r}\bar B_K$-module (with trivial $\bar u_{r}$-
action) and we get \cite{AW}, Lemma 3.2
\begin{equation} \label{twist}
\bar \Delta'_{p^r} (V^{(r)}) \simeq \bar \Delta_p(V)^{(r)}.
\end{equation}
If $L$ is a $\bar U_K$-module, then for all $\bar u_{r} \bar B_K$-modules $M$ we have the tensor identity
\begin{equation} \label{tensor id}
\bar \Delta'_{p^r} (L \otimes M) \simeq L \otimes \bar \Delta'_{p^r} (M) .
\end{equation}
In particular, (\refeq{twist}) and (\refeq{tensor id}) show what happens, when we apply the  $\bar \Delta'_{p^r}$ to a simple $\bar u_{r} B_K$-module:
\begin{equation} \label{simple}
\bar \Delta'_{p^r} (\tilde L_{p^r}(\lambda)) = \bar L_p(\lambda^0) \otimes \bar  \Delta_p(\lambda^1)^{(r)} \text { for all } \lambda \in X.
\end{equation}
Here $\lambda^0 \in X_{p^r}$ and $\lambda^1 \in X$ are as usual determined by the equation $\lambda = \lambda^0 + p^r \lambda^1$.

Note also that 
\begin{equation} \label{comp}
\bar \Delta'_{p^r} \tilde \Delta_{r} (\lambda) =\bar  \Delta_p(\lambda).
\end{equation}

Applying $\bar \Delta'_{p^r}$ to a composition series for the $\bar u_{r} B_K$-module $\tilde \Delta_{r}(\lambda)$ we obtain by combining (\refeq{simple}) and (\refeq{comp})

\begin{prop} \label{p-filt} Let $\lambda \in X$. Then $\bar \Delta_p(\lambda) $ has a filtration 
$$ 0 = F_0 \subset F_1 \subset \cdots \subset F_n = \bar \Delta_p(\lambda)$$
with $F_i/F_{i-1} \simeq \bar L_p(\mu_i^0) \otimes \bar \Delta_p(\mu_i^1)^{(r)}$ for some $\mu_i \in X$, $ i=1, 2, \cdots , n$.
\end{prop}

We call a filtration as the one given in this theorem a $(p^r,\Delta)$-filtration.  If $M$ has such a filtration then we denote by  $(M:L_p(\mu^0) \otimes \bar \Delta_p(\mu^1)^{(r)})$ the number of times a given quotient $\bar L_p(\mu^0) \otimes \bar \Delta_p(\mu^1)^{(r)}$ occurs in this filtration. The way we obtained the filtration of $\bar \Delta_p(\lambda)$ implies
\begin{cor}
Let $\lambda, \mu \in X$. Then $(\bar \Delta_p(\lambda):\bar L_p(\mu^0) \otimes \bar \Delta_p(\mu^1)^{(r)}) = [\tilde \Delta_{r} (\lambda): \tilde L_{p^r}(\mu)]$.
\end{cor}
Here we use the notation $[M:\tilde L_{p^r}(\mu)]$ for the composition factor multiplicity of $\tilde L_{p^r}(\mu)$ in  a $\bar u_{r} \bar B_K$-module $M$.

Employing now also Theorem \ref{SLP_r} we get the following linkage principle for $(p^r,\Delta)$-filtrations of Verma modules in $\bar {\mathcal O}_p$.

\begin{thm} \label{SLP_p} Let $\lambda, \mu \in X$ and $r \geq 1$. If $(\bar \Delta_p(\lambda):\bar L_{p}(\mu^0) \otimes \bar \Delta_p(\mu^1)^{(r)}) \neq 0$ then $\mu \uparrow \lambda$.
\end{thm}

\subsection{Linkage principles for Verma modules in $\mathcal O_q$}

We assume in this section that $q \in K$ has odd order $\ell > 1$ (and $\ell$ is not divisible by $3$ if $R$ contains a component of type $G_2$) and that $p > 0$ (the $p=0$ case was treated in \cite{AM}, Section 3).  We can then  argue just as in Section 4.1:

Let $M$ be a finite dimensional $u_qB_q$-module. As always we assume that $M$ is also a weight module. Then we define
$$\Delta'_q(M) = U_q \otimes_{u_qB_q} M.$$
If $M = N^{[\ell]}$ for some $\bar U_K$-module $N$ then we have, see \cite{AW} Lemma 3.2,
\begin{equation} \label{twist_q}
\Delta'_q(N^{[\ell]}) \simeq \bar \Delta_p(N)^{[\ell]}.
\end{equation}
If $L$ is a $U_q$-module, then the tensor identity gives
\begin{equation} \label{tensor id_q}
 \Delta'_q(L \otimes M) \simeq L \otimes \Delta'_q(M).
\end{equation}
Recall from Section 3.2 the definition of baby-Verma modules. Transitivity of induction gives
\begin{equation} \label{trans_q}
 \Delta'_q (\tilde \Delta_\ell (\lambda)) \simeq \Delta_q(\lambda) \text { for all } \lambda \in X.
 \end{equation}

Consider the case where $M$ is irreducible, i.e.  $M = \tilde L_\ell(\lambda)$ for some $\lambda$ in $X$. As usual we write $\lambda = \lambda^0 + \ell \lambda^1$ with $\lambda^0 \in X_\ell$ and $\lambda^1 \in X$. Then $\tilde L_\ell(\lambda) = L_q(\lambda^0) \otimes \ell \lambda^1$. So by (\refeq{tensor id_q}) and (\refeq{twist_q}) we get
\[
\Delta'_q(\tilde L_\ell(\lambda)) \simeq L_q(\lambda^0) \otimes \bar \Delta_p (\lambda^1)^{[\ell]}
\]
As $\tilde L_\ell(\lambda)$ is a quotient of $\tilde \Delta_\ell(\lambda)$ we have by (\refeq{trans_q})  that $\Delta'_q(\tilde L_\ell(\lambda))$ is a quotient of $\Delta_q(\lambda)$. It follows that $\Delta'_q(M) \in \mathcal O_q $ for all $M$.

Combining the above we obtain
\begin{prop} \label{ell filt}
Let $\lambda \in X$. Then $\Delta_q(\lambda)$ has a finite filtration in $\mathcal O_q$
$$ 0 = F_0 \subset F_1 \subset \cdots \subset F_r = \Delta_q(\lambda)$$
with $F_i/F_{i-1} \simeq  L_q(\mu_i^0) \otimes  \bar \Delta_p(\mu_i^1)^{[\ell]}$ for some $\mu_i \in X$, $ i=1, 2, \cdots , n$.
\end{prop}

A filtration like the one in this proposition is called an $(\ell,\Delta)$-filtration. Note that  for any finite dimensional $u_qB_q$-module $M$  and any $\mu \in X$  the number of times the quotient $L_q(\mu^0) \otimes \bar \Delta_p(\mu^1)^{[\ell]}$ occurs in an $(\ell,\Delta)$-filtration of $\Delta'_q(M)$ equals the multiplicity $[M:\tilde L_\ell(\mu)]$ of the composition factor $\tilde L_\ell(\mu)$ in $M$. Taking $M = \tilde \Delta_\ell (\lambda)$ we obtain    
\begin{cor}
Let $\lambda, \mu \in X$. Then
$$ (\Delta_q(\lambda):L_q(\mu^0) \otimes \bar \Delta_p(\mu^1)^{[\ell]} ) = [\tilde \Delta_\ell(\lambda):\tilde L_\ell(\mu)]$$
for all $\mu \in X$. 
\end{cor}

Just like in the modular case (Theorem \ref{SLP_p} with $ r = 1$) we have

\begin{thm} Let $\lambda, \mu \in X$. If $\tilde L_{\ell}(\mu)$ is a composition factor of $\tilde \Delta_{\ell}(\lambda) $ then $\mu \uparrow \lambda$.
\end{thm}
Of course the strong linkage relation here is with respect to $\ell$ instead of $p$.

This gives

\begin{cor}  Let $\lambda, \mu \in X$. If $(\Delta_q(\lambda):L_q(\mu^0) \otimes \bar \Delta_p (\mu^1)^{[\ell]}) \neq 0$ then $\mu \uparrow \lambda$.
\end{cor}

Consider now $M \in \mathcal O_q$ and assume $M$ has an $(\ell,\Delta)$-filtration. For each quotient $L_q(\mu^0) \otimes \bar \Delta_p(\mu^1)^{[\ell]}$, which occurs in this filtration we choose now a $(p^r,\Delta)$-filtration of $\bar \Delta_p(\mu^1)$. This leads to a refined filtration of $M$ in which quotients have the form $L_q(\mu) \otimes \bar L_p(\nu)^{[\ell]} \otimes (\bar \Delta_p(\eta)^{(r)})^{[\ell]}$ with $ \mu \in X_\ell, \nu \in X_{p^r}$, and $\eta \in X$.  We call such a filtration an $(\ell,p^r,\Delta)$-filtration of $M$. 

Combining Corollary 4.3 and Corollary 4.6 we then get (in the obvious notation)

\begin{cor} \label{ell-p filt} Let $\lambda, \eta \in X$ and $\mu \in X_\ell, \nu \in X_{p^r}$. The Verma module $\Delta_q(\lambda)$ has an $(\ell,p^r,\Delta)$-filtration in which the multiplicities are given by
$$(\Delta_q(\lambda):L_q(\mu) \otimes \bar L_p(\nu)^{[\ell]} \otimes (\bar \Delta_p(\eta)^{(r)})^{[\ell]}) = \sum_{\zeta \in X} [\tilde \Delta_\ell(\lambda):\tilde L_q(\mu + \ell \zeta)][\tilde \Delta_{r}(\zeta):\tilde L_{p^r}(\nu + p^r \eta)].$$
(Note that this sum is finite.)
\end{cor}

\begin{rem} \label{dual}
In the ordinary category $\mathcal O$ as well as in the characteristic zero category $\mathcal O_q$ we have duality functors, see \cite{Hu}, Section 3.3, respectively \cite{AM}, Section 3.8. The very same recipe gives duality functors $D_p$ and $D_q$ on $\bar {\mathcal O}_p$ and $\mathcal O_q$. These functors preserve characters  so we have in particular $ D_p(\bar L_p(\lambda)) \simeq \bar L_p(\lambda)$  and $D_q(L_q(\lambda)) \simeq  L_q(\lambda)$ for all $\lambda \in X$.

We shall denote the dual Verma modules by $\overline \nabla_p(\lambda) = D_p(\bar \Delta_p(\lambda)) $ and $\nabla_q(\lambda) = D_q(\Delta_q(\lambda))$. Then all the above results have dual analogues. In particular, $\bar \nabla_p(\lambda)$ has for each $r$ a $(p^r,\nabla)$-filtration, and $\nabla_q(\lambda)$ has an $(\ell,\nabla)$-filtration satisfying the dual analogues of Theorem 4.4 and Corollary 4.8.

\end{rem}

\section{$\mathfrak g = \mathfrak {sl}_2$}\label{sec sl_2}
In this section we illustrate our results in the previous sections in  the $\mathfrak {sl}_2$-case.

\subsection{The simple modules in $\bar {\mathcal O}_p = \bar {\mathcal O}_p(\mathfrak {sl}_2$)} \label{sl_2 p}
We have $X = \Z$ and $X^+ = \Z_{\geq 0}$. So the simple modules in $\bar {\mathcal O}_p$ are $\bar L_p(n)$ with $n \in \Z$, and among these the ones with $n \geq 0$ are the finite dimensional ones. 

Let $V$ denote the $2$-dimensional natural module for $\mathfrak {sl}_2$. Then it is well known that $\bar L_p(n) = S^nV$ for $n = 0, 1, \cdots , p-1$,
and we get 
\begin{equation}
\Char \bar L_p(n) = e^n + e^{n-2} + \cdots + e^{-n} \text { for all } n = 0, 1, \cdots , p-1.
\end{equation} 

Now let $n \geq 0$ be arbitrary and write $n = n_0 + n_1 p + \cdots + n_r p^r$ with $0 \leq n_i \leq p-1$ for all $i$. Then Theorem {\ref{STP_p}} says
$\bar L_p(n) \simeq \bar L_p(n_0) \otimes \bar L_p(n_1)^{(1)} \otimes \cdots \otimes \bar L_p(n_r)^{(r)}$. This implies $\dim \bar L_p(n) = (n_0+1)(n_1+1)\cdots (n_r+1)$ and
\begin{equation} \label{fd simple}
\Char \bar L_p(n) = \prod_{i=0}^r (e^{n_i} + e^{n_i - 2} + \cdots  + e^{-n_i})^{(i)}.
\end{equation} 

Now consider $n=-1$. An easy direct computation (alternatively use the arguments for (5.5) in the following section)  shows that $\bar L_p(-1) = \bar \Delta_p(-1)$. Hence
\begin{equation}\label{-1}
\Char \bar L_p(-1) = e^{-1} +  e^{-3} +  e^{-5} + \cdots.
\end{equation} 
In this case Proposition \ref{p-filt} gives $\bar \Delta_p(-1) \simeq \overline {St}_{r} \otimes \bar \Delta_p(-1)^{(r)}$ for all $r >0$.

Finally, consider $n < -1$. We choose $r$ so big that $p^r > -n$. Then $n$ has p-adic expansion $n = (p^r + n) + (-1)p^r$ and Theorem \ref{STP_p} implies $\bar L_p(n) \simeq \bar L(p^r+n) \otimes \bar L_p(-1)^{(r)}$. Therefore we get
\begin{equation}
\Char \bar L_p(n) = \Char \bar L_p(p^r +n) (\Char \bar L_p(-1))^{(r)}.
\end{equation}
Note that (\refeq{fd simple}) and (\refeq{-1}) determine the two characters on the right hand side.

\subsection{The Verma modules in $\bar {\mathcal O}_p$}

Let $n \in \Z$. Proposition \ref{p-filt} gives
\begin{enumerate}
\item If $n \equiv -1 (\mod p)$ then $\bar \Delta_p(n) \simeq \overline {St}_1 \otimes \bar \Delta_p(\frac{n+1}{p} -1)^{(1)}$,
\item if $n \not \equiv -1 (\mod p)$ we write $n= n_0 + n_1 p$ with $0 \leq n_0 < p-1$ and get a short exact sequence
$$ 0 \rightarrow \bar L_p(p-n_0-2) \otimes \bar \Delta_p(n_1-1)^{(1)} \rightarrow \bar \Delta_p(n) \rightarrow \bar L_p(n_0) \otimes \bar \Delta_p(n_1)^{(1)} \rightarrow 0.$$
\end{enumerate}
This allows us to determine the structure of $\bar \Delta_p(n)$ for all $n$. For instance, we get
\begin{equation}
\bar \Delta_p(-1) = \bar L_p(-1).
\end{equation}
In fact, for $n=-1$ the isomorphism in (1) says $\bar \Delta_p(-1) \simeq \overline {St}_1 \otimes \bar \Delta_p(-1)^{(1)}$.
 Repeated use of (1) on the second tensor factor here leads to $\bar \Delta_p(-1) \simeq \overline {St}_r \otimes \bar \Delta_p(-1)^{(r)}$. 
 On the other hand, Steinberg's tensor product theorem, cf. Corollary 3.8, gives 
 $\bar L_p(-1) \simeq St_r \otimes \bar L_p(-1)^{(r)}$. 
 This implies that the weight spaces of $\bar \Delta_p(-1)$ and $\bar L_p(-1)$ coincide.

When $n=0$ the short exact sequence (2) reads (noting that by Steinberg's tensor product theorem $\bar L_p(p-2) \otimes  \bar L_p(-1)^{(1)} \simeq \bar L_p(-2)$)
\begin{equation} 0 \rightarrow \bar L_p(-2) \rightarrow \bar \Delta_p(0) \rightarrow \bar \Delta_p(0)^{(1)} \rightarrow 0.
\end{equation}
Repeated use of this sequence shows that for any $r>0$ we have a short exact sequence
\begin{equation} \label{ses -2} 0 \rightarrow M_r \rightarrow \bar \Delta_p(0) \rightarrow \Delta_p(0)^{(r)} \rightarrow 0,
\end{equation}
where the submodule $M_r$ has a finite composition series with factors $\bar L_p(-2), \bar L_p(-2)^{(1)}, \cdots , \bar L_p(-2)^{(r-1)}$. 
In particular, we see that $\bar \Delta_p(0)$ has simple socle equal to $\bar L_p(-2)$. In fact, similar computations give for general $n\geq 0$, 
that the socle of $\bar \Delta_p(n)$ is $\bar L_p(-n-2)$.

Likewise, we get for $n = -2$ the short exact sequence
\begin{equation}
0 \rightarrow \bar \Delta_p(-2)^{(1)} \rightarrow \bar \Delta_p(-2) \rightarrow \bar L_p(-2) \rightarrow 0.
\end{equation}
This leads for each $r>0$ to the following short exact sequences
$$0 \rightarrow \bar \Delta_p(-2)^{(r)} \rightarrow \bar \Delta_p(-2) \rightarrow N_r \rightarrow 0,$$
where the quotients $N_r$ have finite composition series with factors 
$\bar L_p(-2)^{(r-1)}, \bar L_p(-2)^{(r-2)}, \cdots ,\\ \bar L_p(-2)$. 
In particular, we see from this that $\Hom_{\bar {\mathcal O}_p}(\bar L_p(n), \bar \Delta_p(-2)) = 0$ for all $n$. 
Again this result holds also when we replace $-2$ by any smaller integer, cf. Corollary 8.2 below for general $\mathfrak g$. 

\subsection{The simple modules in $\mathcal O_q = \mathcal O_q(\mathfrak {sl}_2$)}
In this and the next section we consider the quantum $\mathfrak {sl}_2$-case where $q \in K$ is a primitive root of unity of order $\ell >1$. As always we assume for convenience that $\ell $ is odd.

We have analogues of the results in Section 5.1. Skipping further details we only state the following results (which determine all irreducible characters in $\mathcal O_q$):
\begin{equation} \label{n small}
\Char L_q(n) = e^n + e^{n-2} + \cdots + e^{-n} \text { for  } n = 0, 1, \cdots , \ell -1,
\end{equation}
\begin{equation} \label{n positive}
\Char L_q(n) = \Char L_q(n_0) (\Char \bar L_p(n_1))^{(\ell)} \text { for all }  n= n_0 + n_1 \ell \text { with } 0 \leq n_0 \leq \ell -1, n_1 \in \Z,
\end{equation}
\begin{equation}
 L_q(-1) = \Delta_q(-1) \text { and } \Char L_q(-1) = e^{-1} + e^{-3} + e^{-5} + \cdots .
\end{equation}

Note that the two characters appearing on the right hand side in (\refeq{n positive}) are given by (\refeq{n small}), respectively  the formulas in Section \ref{sl_2 p}.

\subsection{The Verma modules in $\mathcal O_q$}
We list the following results analogous to those in Section 5.2.

There is an exact sequence
\begin{equation}
 0 \rightarrow L_q(-2) \rightarrow \Delta_q(0) \rightarrow \bar \Delta_p(0)^{[\ell]} \rightarrow 0.
\end{equation}
Combining this with (\refeq{ses -2}) we see that $\Delta_q(0)$ has for each $r \geq 0$ a submodule $M_r$ with composition factors $L_q(-2), L_q(-2\ell), L_q(-2\ell p), \cdots , L_q(-2\ell p^{r-1})$. The quotient of $M_r$ is $(\bar \Delta_p(0)^{(r)})^{[\ell]}$. In particular, we see that $\bar \Delta_p(0)$ has simple socle. This fact remains true, when we replace $0$ by any other non-negative integer, cf. Corollary 8.5 for general $\mathfrak g$.

Likewise, we have a short exact sequence
\begin{equation}
0 \rightarrow  \Delta_q(-2)^{[\ell]} \rightarrow \Delta_q(-2) \rightarrow L_q(-2) \rightarrow 0.
\end{equation}

This leads for each $r\geq 0$ to a quotient $N_r =\Delta_q(-2)/ (\bar \Delta_p(-2)^{(r)})^{[\ell]}$ with composition factors 
$ L_q(-2\ell {p^{r-1})}, L_q(-2 \ell {p^{r-2}}), \cdots , \bar L_q(-2)$. 

 In particular, we see that $\Hom_{ {\mathcal O}_q}(L_q(n), \Delta_q(-2)) = 0$ for all $n$. 
This result holds also when we replace $-2$ by any smaller integer.

\section{The special Verma module}
We return to the case of a general Cartan matrix, but continue to assume that $q$ is an $\ell$'th  root of unity in a field $K$ of characteristic $p > 0$. First we treat the case $\ell = 1$.

\subsection{The special Verma module in $\bar {\mathcal O}_p$}

\begin{thm} \label{irr Verma}
Let $\lambda \in X$. Then $\bar \Delta_p(\lambda)$ is irreducible if and only if $\lambda = - \rho$.
\
\end{thm}

\begin{proof} To prove that $\bar \Delta_p(-\rho)$ is irreducible we show that the surjection $\bar \Delta_p(-\rho) \rightarrow \bar L_p(-\rho)$ is an isomorphism. This follows once we check that 
\begin{equation} \label{spec Verma weights}
\dim_K \bar \Delta_p(-\rho)_{-\rho - \mu} \leq \dim_K \bar L_p(-\rho)_{-\rho - \mu}
\end{equation}
for all $\mu \geq 0$. 
By the modular analogue of \refeq{Verma char} the left hand side of \ref{spec Verma weights} equals $P(\mu)$. By Corollary \ref{iterate} we have $\bar L_p(-\rho) \simeq \overline {St}_{r} \otimes \bar L_p(-\rho)^{(r)}$ for all $r > 0$. Hence $\dim_K \bar L_p(-\rho)_{-\rho - \mu} \geq  \dim_K (\overline {St}_{r})_{(p^r - 1)\rho - \mu}$. 
By the analoque of (\ref{baby char}) this last dimension equals $ \# \{(n_\beta)_{\beta \in R^+} \mid 0 \leq n_\beta < p^r \text { for all } \beta \text { and } \sum_{\beta \in R^+} n_\beta \beta = \mu \}$. If $r$ is large enough then the conditions $n_\beta < p^r$ are automatically  satisfied for all $(n_\beta))$ for which $ \sum_\beta n_\beta \beta = \mu$. Hence  this number equals $P(\mu)$ for all $r \gg 0$.

To prove the converse consider $\lambda \neq -\rho$ and write as usual $\lambda =\lambda^0 + p^r \lambda^1$. 
Then there exists $r > 0$ such that $\lambda^0 
\neq (p^r-1)\rho$ and by the analogue of Proposition \ref{irr baby Verma} we have that $\tilde \Delta_{r}(\lambda)$ is not irreducible. 
Therefore the $(p^r,\Delta)$-filtration of $\bar \Delta_p(\lambda) $ from Proposition \ref{p-filt} contains more than one term and thus $\bar \Delta_p(\lambda)$ is not simple.

\end{proof}
This result allows us in particular to give the following formula for the character of a simple module in $\bar {\mathcal O}_p$ with antidominant highest weight (recall from Section 2.3 the notation $q^- =\sum_{\mu \leq -\rho} P(-\rho - \mu) e^\mu $).
\begin{cor} \label{antidominant simple}
Let $\lambda \in X^-$ and choose $r$ so big that $p^r\rho + \lambda \in X_r$. Then
$$ \Char \bar L_p(\lambda) = \Char\bar  L_p(p^r \rho +\lambda) (q^-)^{(r)}.$$
\end{cor}

\begin{proof} When we combine Theorem \ref{STP_p} with Theorem \ref{irr Verma} we get $\bar L_p(\lambda) \simeq \bar L_p(p^r\rho + \lambda) \otimes \bar \Delta_p(-\rho)^{(r)}$. The formula then follows from (\refeq{Verma char}).
\end{proof}

The special Verma module $\bar \Delta_p(-\rho)$ has also other nice properties:

\begin{thm} \label{spec proj_p}
The Verma module $\bar \Delta_p(-\rho)$ is projective and injective in $\bar {\mathcal O}_p$.
\end{thm}

\begin{proof}
Recall that since  $\bar \Delta_p(-\rho)$ is simple, it is self-dual. Hence it is enough to check that it is projective.

Suppose $M \rightarrow N$ is a surjection in $\bar {\mathcal O}_p$. We claim that the induced map
\begin{equation} \label{surj}
\Hom_{\bar {\mathcal O}_p}(\bar \Delta_p(-\rho), M) \rightarrow  \Hom_{\bar {\mathcal O}_p}(\bar \Delta_p(-\rho), N)
\end{equation}
is also surjective.

The definition of Verma modules implies that for any $V \in \bar {\mathcal O}_p$ we have
\begin{equation} \label{Verma prop}
\Hom_{\bar {\mathcal O}_p}(\bar \Delta_p(-\rho), V) \simeq \{v \in V_{-\rho} \mid \bar B_K v = K v \}.
\end{equation}

To check the surjectivity of (\refeq{surj}) we consider first the special case, where $M$ has no weights strictly bigger than $-\rho$. In that case (\refeq{Verma prop}) gives $\Hom_{\bar {\mathcal O}_p}(\bar \Delta_p(-\rho), M) = M_{-\rho}$ and $\Hom_{\bar {\mathcal O}_p}(\bar \Delta_p(-\rho), N) = N_{-\rho}$. Hence in this case the claim certainly holds.

So consider now a general $M \in \bar {\mathcal O}_p$. We can find weights $\lambda_1, \lambda_2, \cdots , \lambda_n$ such that all weights $\mu$ of $M $ satisfy $\mu \leq \lambda_i$ for some $i$. By Theorem \ref{irr Verma} and  Theorem \ref{STP_p} we have $\bar \Delta_p(-\rho) = St_r \otimes \bar \Delta_p(-\rho)^{(r)}$. By Proposition \ref{homid} (2) this gives
\begin{equation} \label{Steinberg red}
\Hom_{\bar {\mathcal O}_p}(\bar \Delta_p(-\rho), M)  \simeq \Hom_{\bar {\mathcal O}_p}(\bar \Delta_p(-\rho), \Hom_{\bar u_{r}}(St_r, M)^{(-r)}).
\end{equation}
Now the natural map $St_r \otimes \Hom_{\bar u_{r}}(St_r, M) \rightarrow M$ is injective and hence for  any weight $\mu$ of $ \Hom_{\bar u_{r}}(St_r, M)^{(-r)}$ we must have that $(p^r-1)\rho + p^r \mu$ is a weight of $M$. This implies that  $(p^r-1)\rho + p^r \mu \leq \lambda_i$ for some $i$. So if $r$ is large enough we cannot have $\mu$ strictly bigger than $-\rho$. Combining this with (\refeq{Steinberg red}) we then deduce the general case from the special case treated above.

\end{proof}

\subsection{The special Verma module in $\mathcal O_q$}

In analogy with Theorem \ref{irr Verma} we get
\begin{thm} \label{irr Verma_q}
Let $\lambda \in X$. Then $\Delta_q(\lambda)$ is irreducible if and only if $\lambda = -\rho$.
\end{thm}

\begin{proof} 
As $\tilde \Delta_q((\ell -1)\rho) = St_q$ we see from Proposition \ref{ell filt} that $\Delta_q(-\rho) \simeq St_q \otimes \bar \Delta_p(-\rho)^{[\ell]}$. By Theorem \ref{STP_q} we get on the other hand $L_q(-\rho) \simeq St_q \otimes \bar L_p(-\rho)^{[\ell]}$. Hence it follows from Theorem \ref{irr Verma} that $\Delta_q(- \rho)$ is irreducible.

The proof of the converse is completely analogous to the proof given in Theorem \ref{irr Verma}.
\end{proof}

\begin{thm} \label{proj_q}
The Verma module $\Delta_q(-\rho)$ is projective and injective in $\mathcal O_q$.
\end{thm}
\begin{proof} 
As in the proof of Theorem \ref{irr Verma_q} we have  $\Delta_q(-\rho) = St_q \otimes \bar \Delta_p(-\rho)^{[\ell]}$. 
Therefore, by Proposition \ref{homid} (1) we get
 $$\Hom_{\mathcal O_q}(\Delta_q(-\rho), M) \simeq \Hom_{\bar {\mathcal O}_p}(\bar \Delta_p(-\rho), \Hom_{u_q}(St_q, M)^{[-\ell]})$$ 
 for all $M \in \mathcal O_q$. 
 Now $\Hom_{u_q}(St_q, -)$ is exact (e.g. by the $q$-analogue of \cite{RAG}, Lemma II.11.8),
 and according to Theorem \ref{spec proj_p} so is $\Hom_{\bar {\mathcal O}_p}(\bar \Delta_p(-\rho), -)$. 
\end{proof}

\section{Tilting modules}
We have both finite dimensional tilting modules and infinite dimensional ones in our BGG-categories.
First we consider the latter class, which we shall call $\infty$-tilting modules. We  prove their most basic properties. Then we recall a few facts about the better studied class of finite dimensional tilting modules. We show how these facts (under some restrictions on $p$) make us able to obtain an explicit relation between the indecomposable $\infty$-tilting modules and certain indecomposable finite dimensional tilting modules.

\subsection{$\infty$-tilting modules}
Our results in this section hold true in $\mathcal O_q$ as well as in $\bar {\mathcal O}_p$  and the proofs are completely analogous. Therefore we formulate and prove our statements only in the first case.

Recall from Remark \ref{dual} that we have a duality functor $D_q$ on $\mathcal O_q$ (as well as a similar duality $D_p$ on $\bar {\mathcal O}_p$), and that for each $\lambda \in X$ we define $\nabla_q(\lambda) = D_q (\Delta_q(\lambda))$ .

\begin{defn} A module $Q$ in $\mathcal O_q$ is called an $\infty$-tilting module if it has a finite $\Delta$-filtration as well as a finite $\nabla$-filtration.
\end{defn}
We emphasize that in this paper by $\Delta$-filtration we mean a filtration in which each successive quotient is a Verma module. If we replace Verma by Weyl then we get the finite dimensional tilting modules instead.

\begin{exa} We have already come across one $\infty$-tilting module, namely the special Verma module $\Delta_q(-\rho)$, which we studied in the previous section. By Theorem \ref{irr Verma_q} we have $\Delta_q(-\rho)) \simeq L_q(-\rho) \simeq \nabla_q(-\rho)$.
\end{exa}
\vskip .3 cm
If $Q$ has a finite $\Delta$-filtration we shall denote the multiplicity with which a given Verma module $\Delta_q(\lambda)$ occurs as a quotient in such a filtration by $(Q:\Delta_q(\lambda))$. Note that the head of $Q$  is contained in $\oplus L_q(\lambda)$ with $\lambda$ occurring with multiplicity $(Q:\Delta_q(\lambda))$ in the sum. In particular, the head of $Q$ consists of finitely many simple modules. It follows that $Q$ splits into a finite number of indecomposable summands. Each of these summands have a $\Delta$-filtration. Hence if $Q$ is an $\infty$-tilting module it follows that it splits into a finite number of indecomposable $\infty$-tilting modules. 

\begin{lem} \label{finite tensor} Suppose $Q$ has a finite $\Delta$-filtration and let $F$ be a finite dimensional module.  Then $F \otimes Q$ has a finite $\Delta$-filtration,  and $(F \otimes Q: \Delta_q(\lambda)) = \sum_\mu (Q:\Delta_q(\mu)) \dim_K F_{\lambda - \mu}$ for all $\lambda$.
\end{lem}

\begin{proof} It is clearly enough to check that $F \otimes \Delta_q(\lambda)$ has a $\Delta$-filtration for all $\lambda \in X$. This is seen by the standard argument:  Observe that $F \otimes \Delta_q(\lambda) \simeq \Delta_q(F \otimes \lambda)$. Now take a $B_q$-filtration of $F$ with $1$-dimensional quotients, tensor this filtration by $\lambda$ and apply the exact functor $\Delta_q$ to get a $\Delta$-filtration of $F \otimes \Delta_q(\lambda)$. This construction also proves the multiplicity formula.
\end{proof}

\begin{rem} \label{finite remark} It follows immediately from this lemma that if $Q$ is $\infty$-tilting and $F$ is a finite dimensional self-dual module then $Q \otimes F$ is also $\infty$-tilting.
\end{rem}

\begin{thm} \label{tilt exist}
If $\lambda \in X^+ -\rho$ then there exists a unique (up to isomorphism) indecomposable $\infty$-tilting module $T_q^\infty(\lambda)$ in $\mathcal O_q$ with $T^\infty_q(\lambda)_\lambda \simeq K$ and $T_q^\infty(\lambda)_\mu \neq 0$ if and only if $\mu \leq \lambda$.
\end{thm}

\begin{proof} Let $\lambda \in X^+ - \rho$. Then $\lambda + \rho \in X^+$ and hence $L_q(\lambda + \rho)$ is  finite dimensional. When we combine Lemma \ref{finite tensor} (and Remark \ref{finite remark}) with Theorem \ref{irr Verma_q} we get that $L_q(\lambda + \rho) \otimes \Delta_q(-\rho)$ is $\infty$-tilting. Moreover, 
 $L_q(\lambda + \rho) \otimes \Delta_q(-\rho)$ has a unique indecomposable summand with highest weight $\lambda$ because $\lambda$ is the maximal weight of  $L_q(\lambda + \rho) \otimes \Delta_q(-\rho)$ and it occurs with multiplicity $1$. So we define $T_q^\infty(\lambda)$ to be that summand and have thus proved the existence part of the theorem.
 
 To establish the uniqueness of $T_q^\infty(\lambda)$ we first observe that since $\Delta_q(-\rho)$ is projective and injective in $\mathcal O_q$ (according to Theorem \ref{proj_q}) 
 so is $T_q^\infty(\lambda)$. Now assume that $Q$ is another indecomposable tilting module in $\mathcal O_q$ with highest weight $\lambda$. 
 Then both $T_q^\infty(\lambda)$ and $Q$ projects onto $\nabla_q(\lambda)$, 
 and therefore there is a homomorphism $\phi : T_q^\infty(\lambda) \to Q$ which is non-zero
on the $\lambda$ weight space. 
Likewise, $\Delta_q(\lambda)$ has inclusions into both $T_q^\infty(\lambda)$ and $Q$. The injectivity of $T_q^\infty(\lambda)$ gives a homomorphism $\psi : Q \to T_q^\infty(\lambda)$ which is also non-zero on the $\lambda$ weight space.
We claim that the composite $\psi \circ \phi$ is an isomorphism so that $T_q^\infty(\lambda)$ is a summand of $Q$ (and hence isomorphic to $Q$).
 
 By construction $\psi \circ \phi$ acts as a non-zero scalar $a$ on the $\lambda$ weight space. 
 We will prove that $f = \psi \circ \phi - a  \Id_{T_q^\infty(\lambda)}$ is nilpotent.  This implies that $\psi \circ \phi$ is an isomorphism.
 
 For each $\mu \in X$ we denote by $f_\mu$ the restriction of $f$ to $T_q^\infty(\lambda)_\mu$. 
 Then $f_\lambda = 0$ so that $f$ is certainly not an isomorphism. 
 For any $\mu$ we have that $f_\mu$ is an endomorphism of a finite dimensional vector space and therefore the sequence $\im f_\mu \supset \im f_\mu^2 \supset \cdots $ must stablize, 
 i.e. there exists $n(\mu) \in \Z_{>0}$ such that $\im f_\mu^{n(\mu)} = \im f_\mu^{n(\mu) +1} = \cdots $. 
 We denote now by $\mu_1, \mu_2, \cdots \mu_r$ the highest weights of the Verma modules that occur in a $\Delta$-filtration of $T_q^\infty(\lambda)$ 
 and set $n = \max \{n(\mu_i)\mid i=1, 2, \cdots r\}$. The usual fitting lemma argument then shows that $T_q^\infty(\lambda)_{\mu_i} = \im f_{\mu_i}^n \oplus \Ker f^n_{\mu_i}$ for all $i$.
This implies that the submodule $M = \im f^n \cap \Ker f^n$ of $T_q^\infty(\lambda)$ is $0$.
In fact, by the choice of $n$ we have $M_{\mu_i} = 0$ for all i, and hence $\Hom_{\mathcal O_q}(M, \nabla_q(\mu_i)) = 0$ for all $i$. 
We conclude that $\Hom_{\mathcal O_q}(M, T_q^\infty(\lambda)) = 0$, i.e $M = 0$. 

Using the injectivity of $T_q^\infty(\lambda)$ we get likewise that   $M'= T_q^\infty(\lambda)/(\im f ^n + \ Ker f^n)$ is $0$, because $\Hom_{\mathcal O_q} (\Delta_q(\mu_i), M') = 0$ for all $i$. We conclude that $T_q^\infty(\lambda) = \im f^n \oplus \Ker f^n$. As $T_q^\infty (\lambda)$ is indecomposable it follows that $f^n = 0$.
 \end{proof}

\begin{rem} Note that we cannot expect to extend the result in this theorem to all $\lambda \in X$. Already for $\mathfrak sl_2$ the result fails for $ \lambda = n < -1$. In fact, if all weights of a module in $\mathcal O_q(\mathfrak {sl}_2)$ are less than $-1$ then it cannot be $\infty$-tilting, because  no Verma module $\Delta_q(n)$ with $n<-1$ has a socle, see Section \ref{sec sl_2}.
\end{rem}

\subsection{Finite dimensional tilting modules} We recall some key facts about finite dimensional tilting modules in $\bar {\mathcal O}_p$ and $\mathcal O_q$.  In the case $\bar {\mathcal O}_p$ this class identifies with the tilting modules for the almost simple algebraic group corresponding to $R$. These are well-studied, see e.g. \cite{Do}, \cite{RAG}, Section II.E. In the quantum case the analogous facts hold, see \cite{An92} and \cite{An00}.

Let $h$ denote the Coxeter number for $R$. 
\begin{enumerate}
\item For each  $\mu\in X^+$ there exists a unique (up to isomorphisms) indecomposable finite dimensional tilting module $\bar T_p(\mu)$ with weights $\leq \mu$ and with $\dim_K \bar T_p(\mu)_\mu = 1$, \cite{RAG}, II.E.6.
\item If $r>0$ and $\mu \in X_r$ then $\overline {St}_r \otimes  \bar T_p(\mu)$ is a tilting module with indecomposable summands having highest weights in $\{\nu \in (p^r-1)\rho + X^+ \mid \nu \leq (p^r-1)\rho + \mu \}$, \cite{RAG}, II.E.8
\item Suppose $p \geq 2h-2$. If $\lambda \in X_r$ we set $\hat \lambda = 2(p^r-1)\rho + w_0\lambda$. Then $\bar T_p(\hat \lambda)_{\mid _{\bar u_1}}$ is an indecomposable projective (and injective) $\bar u_1$-module. Its head and socle are isomorpic to $\bar L_p(\lambda)_{\mid _{\bar u_1}}$, \cite {RAG}, II.E.9.

\item For each  $\mu\in X^+$ there exists a unique (up to isomorphisms) indecomposable finite dimensional tilting module $ T_q(\mu)$ with weights $\leq \mu$ and with $\dim_K T_q(\mu)_\mu = 1$, \cite{An92}, Section 4.
\item Note that $St_q = L_q((\ell -1)\rho)$ is tilting, i.e. $St_q = T_q((\ell -1)\rho)$. If $\mu \in X_\ell$ then $St_q \otimes T_q(\mu)$ is a finite dimensional tilting module with indecomposable summands $T_q(\nu)$, where $\nu \in (\ell - 1)\rho + X^+$ and $\nu \leq (\ell - 1)\rho + \mu$, \cite{An00}, Section 4.

\item Suppose $p \geq 2h-2$. If $\lambda \in X_\ell$ and $\hat \lambda = 2(\ell -1)\rho + w_0\lambda$ then $ T_q(\hat \lambda)_{\mid _{\bar u_q}}$ is an indecomposable projective (and injective) $\bar u_q$-module. Its head and socle are isomorpic to $ L_q(\lambda)_{\mid _{\bar u_q}}$, \cite {An01}, Proposition 5.7.
\end{enumerate}

We now prove how (1) - (3) lead to the following result.
\begin{thm} \label {tilt_p rel} Let $\lambda \in X^+ - \rho$ and take $r$ so big that $\lambda + \rho \in X_{p^r}$. Then we have
\begin{enumerate}
\item $\bar T_p(\lambda + p^r\rho) \otimes \bar \Delta_p(-\rho)^{(r)}$ is $\infty$-tilting,
\item Suppose $p \geq 2h-2$. Then $\bar T_p^{\infty}(\lambda) \simeq \bar T_p(\lambda + p^r\rho) \otimes \bar \Delta_p(-\rho)^{(r)}$.
\end{enumerate}
\end{thm}
\begin{proof}
By (2) above we have that $\bar T_p(\lambda + p^r\rho)$ is a summand of $\bar T_p(\lambda + \rho) \otimes \overline {St}_r$. Hence using Theorem \ref{STP_p} we get that
 $\bar T_p(\lambda + p^r\rho) \otimes \bar \Delta_p(-\rho)^{(r)}$ is a summand of $\bar T_p(\lambda + \rho) \otimes \bar \Delta_p(-\rho)$ and the first part of the theorem follows. 
 
Under the assumption $p \geq 2h-2$ we have from (3) above that $\bar T_p(\lambda +p^r \rho)$ has simple socle as $\bar u_{r}$-module. Therefore $ \bar T_p(\lambda + p^r\rho) \otimes \bar \Delta_p(-\rho)^{(r)}$ has simple $\bar U_K$-socle. Hence it is in particular indecomposable. Since it has highest weight $\lambda$ it is therefore  isomorphic to $\bar T_p^\infty(\lambda)$.
 
 \end{proof}

\begin{rem} S. Donkin has conjectured that the assumption $p \geq 2h-2$ in (3) above is unneccessary. However, Bendel, Nakano, Pillen and Sobaje found a counterexample  (for $p=2$ and $R$ of type $G_2$) in \cite{BNPS1}. In a recent preprint \cite {BNPS2} the same authors have verified Donkin's conjecture in several low rank cases and announced further counter examples for $p=2$  (type $B_3$) and $p=3$ (type $C_3$).

Whenever, the conjecture holds we can omit the assumption on $p$ in Theorem \ref{tilt_p rel} (2) as well as in Theorem \ref{tilt_q rel} (2) below. -  See also \cite{BNPS2} for some background on Donkin's conjecture.
\end{rem}

Next we shall prove that (4) - (6) give analogous results in $\mathcal O_q$.
\begin{thm} \label{tilt_q rel} Let $\lambda \in X^+ - \rho$ and take $r$ so big that $\lambda +\rho \in X_{\ell p^r}$. Then
\begin{enumerate}
\item $T_q(\lambda + \ell p^r \rho) \otimes (\bar \Delta_p(-\rho)^{(r)})^{[\ell]}$ is $\infty$-tilting,
\item Suppose $p \geq 2h-2$. Then $ T_q^{\infty}(\lambda) \simeq  T_q(\lambda + \ell p^r \rho) \otimes (\bar \Delta_p(-\rho)^{(r)})^{[\ell]}$.
\end{enumerate}
\end{thm}
\begin{proof}
This proof runs parallel to the proof of Theorem \ref{tilt_p rel}: Note that $\lambda + \rho \in X^+$ and that $\lambda + \ell p^r \rho$ is the maximal weight of $T_q(\lambda + \rho) \otimes St_q \otimes \overline {St}_{r}^{[\ell]}$. This implies that the $\infty$-tilting module $T_q(\lambda + \rho) \otimes \bar \Delta_q(-\rho) \simeq T_q(\lambda + \rho) \otimes St_q \otimes  (\overline {St}_{r} \otimes \bar \Delta_p(-\rho)^{(r)})^{[\ell]}$ contains  $T_q(\lambda + \ell p^r \rho) \otimes (\bar \Delta_p(-\rho)^{(r)})^{[\ell]}$ as a summand. Hence the first part of the theorem holds.

To check the second part it is enough to verify that $T_q(\lambda + \ell p^r \rho) \otimes (\bar \Delta_p(-\rho)^{(r)})^{[\ell]}$ has simple socle. This follows from (6) and (3) above combined with \cite{An01}, Corollary 5.8.

\end{proof}

\section{Projective and injective modules}
We preserve the assumptions on $\ell$ and $p$ from the previous sections. So far we have found one projective module (namely the special Verma module) in each of  the categories $\bar {\mathcal O}_p$ and $\mathcal O_q$. In the present section we use this projective module to obtain projective covers and injective envelopes of all simple modules with antidominant highest weights. We also show that the same cannot be done when the highest weight is not antidominant.

It turns out that the projective covers we obtain are also indecomposable  $\infty$-tilting modules. This leads to a reciprocity law relating the Verma multiplicities in indecomposable $\infty$-tilting modules to multiplicities in the $(p^r,\Delta)$-, respectively $(\ell,p^r,\Delta)$-filtration of the Verma modules.

\subsection{Projectives and injectives in $\bar {\mathcal O}_p$}

Set $X^- = \{\lambda \in X \mid \langle \lambda, \alpha^\vee \rangle \leq -1 \text { for all } \alpha \in S\}$. This is the set of antidominant weights or, more precisely, the closure of the antidominant cone in $X$.
\begin{thm} \label{proj cover}
For each $\lambda \in X^-$ there exists a projective indecomposable module $\bar P_p(\lambda ) \in \bar {\mathcal O}_p$ with head $\bar L_p(\lambda)$. Moreover, $\bar P_p(\lambda)$ is  also the injective envelope of $\bar L_p(\lambda)$ and coincides with the $\infty$-tilting module $\bar T_p^\infty (w_0 \cdot \lambda)$.
\end{thm}
\begin{proof} Let $\lambda \in X^-$. Then we define 
\begin{equation} \label{def proj_p}
\bar P_p(\lambda) = \bar T_p^\infty(w_0 \cdot \lambda).
\end{equation}
The theorem follows once we prove that $ \bar T_p^\infty(w_0 \cdot \lambda)$ is projective and has $\bar L_p(\lambda)$ as a quotient. 

By construction, see the proof of Theorem \ref{tilt exist}, $\bar T_p^\infty(w_0\cdot \lambda)$ is a summand of $\bar L_p(w_0(\lambda + \rho)) \otimes \bar \Delta_p(-\rho)$. Now Theorem \ref{spec proj_p} says that $\bar \Delta_p(-\rho)$ is projective and hence so are $\bar L_p(w_0(\lambda + \rho)) \otimes \bar \Delta_p(-\rho)$ and its summand $\bar T_p^\infty(w_0\cdot \lambda)$.

We claim that $\bar L_p(\lambda)$ is a submodule of $\bar T_p^\infty(w_0\cdot\lambda)$. As $\bar T_p^\infty(w_0\cdot\lambda)$ is selfdual 
this is equivalent to $\bar L_p(\lambda)$ being a quotient. Now the highest weight of $\bar T_p^\infty(w_0\cdot\lambda)$ is $w_0\cdot\lambda$
and therefore $\bar\Delta_p(w_0\cdot \lambda)$ is a submodule. It thus suffices to prove that $\bar L_p(\lambda) \subset \bar \Delta_p(w_0\cdot \lambda)$. 
However, if we choose $r$ so large that $w_0\cdot \lambda \in X_{p^r}$ then $\tilde L_{p^r}(\lambda)$ is the socle of $\tilde \Delta_{r}(w_0\cdot \lambda)$, see \cite{RAG}, Proposition II.9.6.
 It follows that $\bar L(\lambda + p^r\rho) \otimes \bar \Delta(-\rho)^{(r)}$ is the first term in a $(p^r,\Delta)$-filtration of $\bar \Delta_p(w_0\cdot \lambda)$, see Proposition \ref{p-filt}. However, by Theorem \ref{SLP_p} and Theorem \ref{irr Verma} we have $\bar L_p(\lambda + p^r\rho) \otimes \bar \Delta_p(-\rho)^{(r)} \simeq \bar L_p(\lambda)$.
\end{proof}

\begin{cor} If $\lambda \in -\rho + X^+$ then the socle of $\bar \Delta_p(\lambda)$ equals $\bar L_p(w_0 \cdot \lambda)$.
\end{cor} 

\begin{proof} As in the proof of Theorem \ref{proj cover} we get that $\bar \Delta_p(\lambda)$ is a submodule of $\bar P_p(w_0 \cdot \lambda)$. 
\end{proof}

We shall now show that the analogue of Theorem \ref{proj cover} fails for all irreducible modules with highest weight outside $X^-$.
\begin{thm}  If $P$ is a projective module in $\bar {\mathcal O}_p$ then $\Hom_{\bar {\mathcal O}_p}(P, \bar L_p(\lambda))
 = 0$ for all $\lambda \in X\setminus X^-$. Likewise, when  $\lambda \notin X^-$  an injective module $I$ in $\bar {\mathcal O}_p$ cannot contain $\bar L_p(\lambda)$.
\end{thm}

\begin{proof}
Suppose that $P$ is a projective module in $\bar {\mathcal O}_p$. If there is a non-zero map $P \rightarrow \bar L_p(\lambda)$ for some $\lambda \in X$, then there is also for each $r \geq 1$ a non-zero map $P \rightarrow \bar L_p(\lambda^0) \otimes \bar \nabla_p(\lambda^1)^{(r)}$. Set $\hat \lambda = 2(p^r-1)\rho +w_0 \lambda^0 + p^r\lambda^1$. Then  $\tilde L_{p^r}(\lambda)$ is the head of $\tilde \nabla_{p^r}(\hat \lambda)$, see \cite{RAG}, Proposition II.9.6,  and we get that the $(p^r,\nabla)$-filtration of $\bar \nabla_p(\hat \lambda)$ constructed in Proposition\ref{p-filt} contains a surjection $\bar \nabla_p(\hat \lambda) \rightarrow \bar L_p(\lambda^0) \otimes \bar \nabla_p(\lambda^1)^{(r)}$. As $P$ is projective this results in a non-zero map $P \rightarrow \bar \nabla_p(\hat \lambda)$. Therefore $\hat \lambda $ must be a weight of $P$. We shall now prove that this is impossible if $\lambda \notin X^-$ and $r \gg 0$.

As we are going to vary the exponent of $p$ we introduce the following notation (for use in the rest of this proof)
$$ \lambda = \lambda^0(r) + p^r \lambda^1(r) \text { where } \lambda^0(r) \in X_r \text  { and } \lambda^1(r) \in X.$$
$$ \hat \lambda(r) = \widehat {\lambda^0(r)} + p^r \lambda^1(r) \text { where } \widehat {\lambda^0(r)} = 2 (p^r - 1) \rho +w_0 \lambda^0(r)$$

Suppose now $\lambda \in X\setminus X^-$. Then there exists $\alpha \in S$ such that $\langle \lambda, \alpha^\vee \rangle \geq 0$. Choose $r$ so big that $\omega = \lambda^1(r)$ satisfies $\langle \omega, \beta^\vee \rangle \in \{-1, 0\}$ for all $\beta \in S$. We have $\langle \omega, \alpha^\vee \rangle = 0$. Note also that $\omega$ is equal to $\lambda^1(r')$ for all $r'\geq r$.

We have $\omega^0(s) = (p^s-1)(- \omega)$ and $\omega^1(s) = \omega$. It follows that 
$$ \lambda^0(r+s) = \lambda^0(r) + p^r(1-p^s)\omega \text { and } \widehat {\lambda(r+s)} = 2(p^{r+s} -1)\rho + w_0 \lambda^0(r) + p^r(1-p^s)w_0(\omega) + p^{r+s} \omega.$$
We get from this
\begin{equation} \label{alpha}
\langle \hat \lambda(r+s), \alpha^\vee \rangle \geq 2(p^{r+s} -1) -(p^r-1) + p^r(1-p^s) = p^{r+s} - 1 ,
\end{equation}
and
\begin{equation} \label{beta}
\langle \hat \lambda(r+s), \beta^\vee \rangle \geq 2(p^{r+s} -1) -(p^r-1) + p^r(1-p^s)  - p^{r+s} = -1 \text { for all } \beta \in S\setminus \{\alpha\}.
\end{equation}
The inequalities (\refeq{alpha}) and (\refeq{beta}) show that $\hat \lambda(r+s) + \rho \in X^+$ for all $s$. Now since $P$ is a module in $\bar {\mathcal O}_p$ there exist $\mu_1, \mu_2, \cdots , \mu_n$ such that for any weight $\nu$ of $P $ we have $\nu \leq \mu_i$ for some $i$. But suppose $\hat \lambda(r+s) \leq \mu_i$. We have
$$ \mu_i = \sum_{\beta \in S} a_\beta \beta, \; \rho = \sum_{\beta \in S} b_\beta \beta, \text { and } \hat \lambda(r+s) + \rho =    \sum_{\beta \in S} c_\beta(s) \beta$$
for some $a_\beta \in \Q, b_\beta \in \Q_{>0}$ and $ c_\beta(s) \in \Q_{\geq 0}$. Then the inequality $\hat \lambda(r+s) \leq \mu_i$ is equivalent to
$$ c_\beta(s) - b_\beta \leq a_\beta \text { for all } \beta \in S.$$ 
Combining this inequality for $\beta = \alpha$ with (\refeq{alpha}) we get
$$ p^{r+s} \leq \langle \hat \lambda(r+s) + \rho, \alpha^\vee \rangle \leq 2 c_\alpha(s) \leq 2(a_\alpha + b_\alpha).$$
This is clearly impossible when $s \gg 0$.

\end{proof}

\subsection{Projectives and injectives in $\mathcal O_q$} 

The results in Section 8.1 all have direct analogues in $\mathcal O_q$. As the proofs for $\mathcal O_q$ go exactly as for $\bar {\mathcal O}_p$ we only give the statements.

\begin{thm} 
For each $\lambda \in X^-$ there exists a projective indecomposable module $ P_q(\lambda ) \in\mathcal O_q$ with head $ L_q(\lambda)$. Moreover, $P_q(\lambda)$ is self-dual and is therefore also the injective envelope of $ L_q(\lambda)$ and coincides with the $\infty$-tilting module $ T_q^\infty (w_0 \cdot \lambda)$.
\end{thm}

\begin{cor}
 If $\lambda \in -\rho + X^+$ then the socle of $ \Delta_q(\lambda)$ equals $ L_q(w_0 \cdot \lambda)$.
\end{cor}

\begin{thm} \label{no proj_q} If $P_q$ is a projective module in $ {\mathcal O}_q$ then $\Hom_{{\mathcal O}_q}(P_q,  L_q(\lambda))
 = 0$ for all $\lambda \in X\setminus X^-$. Likewise, when  $\lambda \notin X^-$  an injective module $I$ in $ {\mathcal O}_q$ cannot contain $ L_q(\lambda)$.
\end{thm}

\subsection{Reciprocity laws}
The correspondence between the projective covers of simple modules with antidominant highest weights and the indecomposable $\infty$-tilting modules with highest weights in the closure of the dominant chamber gives rise to some reciprocity laws. To deduce these we need the following result (valid by similar arguments also in $\bar {\mathcal O}_p$).

\begin{lem} \label{delta-nabla} Let $\lambda, \mu \in X$. Then
\begin{enumerate}
\item $\Hom_{\mathcal O_q} (\Delta_q(\lambda), \nabla_q(\mu)) \simeq \begin{cases} {K \text { if } \lambda = \mu,} \\ {0 \text { otherwise.}}\end{cases}$
\item $\Ext^1_{\mathcal O_q}(\Delta_q(\lambda), \nabla_q(\mu)) = 0.$
\end{enumerate}
\end{lem}

Note that the $\Ext$ in (2) should be interpreted as Yoneda-Ext (we proved in Theorem \ref{no proj_q} that $\mathcal O_q$ does not have enough projectives/injectives).

\begin{proof}
This goes by the standard arguments:  The universal property of Verma modules (\refeq{Verma prop}) shows that if $\Hom_{\mathcal O_q} (\Delta_q(\lambda), \nabla_q(\mu)) \neq 0$ then $\mu \geq \lambda$. Dually, we get $\lambda \geq \mu$.  Also the statement $\Hom_{\mathcal O_q} (\Delta_q(\lambda), \nabla_q(\lambda)) = K$ follows immediately from  (\refeq{Verma prop}).

Now suppose we have an extension
\begin{equation} \label{split} 0 \rightarrow \nabla_q(\mu) \rightarrow E \rightarrow \Delta_q(\lambda) \rightarrow 0
\end{equation}
in $\mathcal O_q$. Unless $\mu  > \lambda$ we get from  (\refeq{Verma prop}) a homomorphism $\Delta_q(\lambda) \rightarrow E$, which splits (\refeq{split}). On the other hand, if $\mu > \lambda$ then the same argument splits the sequence dual to (\refeq{split}).

\end{proof}

Recall from Section 4.1 the notion of a $(p^r,\Delta)$-filtration for a module $M$ in $\bar {\mathcal O}_p$ and the notation $(M:\bar L_p(\mu^0) \otimes \bar \Delta_p(\mu^1)^{(r)})$ for the number of occurrences of $\bar L_p(\mu^0) \otimes \bar \Delta_p(\mu^1)^{(r)}$ in such a filtration. Note that if $\mu$ is antidominant and $r$ is large then $\bar L_p(\mu^0) \otimes \bar \Delta_p(\mu^1)^{(r)} = \bar L_p(\mu)$, see Corollary \ref{antidominant simple}. Dually we have $(p^r,\nabla)$-filtrations with analogous notation. Of course we have  $(M:\bar L_p(\mu^0) \otimes \bar \Delta_p(\mu^1)^{(r)}) =  (D_pM:\bar L_p(\mu^0) \otimes \bar \nabla_p(\mu^1)^{(r)})$.

 We get the following reciprocity law in $\bar {\mathcal O}_p$.
\begin{thm} \label{reciprocity_p}Let $ \lambda, \mu \in X$ and suppose $\lambda + \rho \in X^+$. Then
$$  (\bar T^\infty_p(\lambda) : \bar  \Delta_p(\mu)) = (\bar \Delta_p(\mu): \bar L_p(w_0\cdot \lambda)).$$
\end{thm}
\begin{proof}
Lemma \ref{delta-nabla} implies that  
$$(\bar T^\infty_p(\lambda) : \bar  \Delta_p(\mu)) = \dim_K \Hom_{\bar {\mathcal O}_p}(\bar T^\infty_p(\lambda), \bar \nabla_p(\mu)).$$ 
By Theorem \ref{proj cover} we may replace $\bar T^\infty_p(\lambda)$ by $\bar P_p(w_0 \cdot \lambda)$. The projectivity of $\bar P_p(w_0 \cdot \lambda)$ gives (when applying $\Hom_{\bar {\mathcal O}_p}( \bar P_p(w_0\cdot \lambda), -)$ to a $(p^r, \nabla)$-filtration of $\bar \nabla_p(\mu)$) the identity
$$  \dim_K \Hom_{\bar {\mathcal O}_p}(\bar P_p(w_0 \cdot \lambda), \bar \nabla_p(\mu)) = (\bar \nabla_p(\mu): \bar L_p(w_0\cdot \lambda)).$$
Here we have chosen $r$ so large that $\lambda + \rho \in X_{p^r}$ (so that $\bar L_p(w_0 \cdot \lambda) = \bar L_p(w_0 \lambda + p^r \rho)  \otimes \bar L_{p^r}(-\rho)^{(r)}$, cf. Corollary \ref{antidominant simple}). 
\end{proof}

In $\mathcal O_q$ we use the $(\ell,p^r,\nabla)$-filtrations of dual Verma modules to obtain the following analogous result.
\begin{thm} Let $ \lambda, \mu \in X$ and suppose $\lambda + \rho \in X^+$. Then
$$  ( T^\infty_q(\lambda) :  \Delta_q(\mu)) = (\Delta_q(\mu): L_q(w_0\cdot \lambda)),$$
where the right hand side denotes the multiplicity of $L_q(w_0 \cdot \lambda)$ in an $(\ell,p^r,\nabla)$-filtration of $\nabla_q(\mu)$, with $r$ chosen so large that $w_0\cdot \lambda + \ell p^r \rho \in X_{\ell p^r}$.
\end{thm}

These reciprocity laws imply the following equivalence.
\begin{cor} \label{last} The set of indecomposable tilting characters $\{\Char \bar T^\infty_p(\lambda) \mid \lambda + \rho \in X^+\}$, respectively $\{\Char T^\infty_q(\lambda) \mid \lambda + \rho \in X^+\}$, determines the set of irreducible characters in $\bar {\mathcal O}_p$, respectively $\mathcal O_q$, and vice versa.
\end{cor}

\begin{proof} Combining Theorem \ref{reciprocity_p} combined with Corollary 4.3 we get 
\begin{equation} \label{last_p} (\bar T^\infty_p(\lambda) : \bar  \Delta_p(\mu)) = [\tilde \Delta_{r}(\mu): \tilde L_{p^r}(w_0\cdot \lambda)].
\end{equation}
for all $\lambda + \rho \in X^+$ and $\mu \in X$. Because of the periodici
ty 
$$ [\tilde \Delta_{r}(\nu): \tilde L_{p^r}(\eta))] =  [\tilde \Delta_{r}(\nu + p^r\rho): \tilde L_{p^r}(\eta + p^r \rho)],$$
valid  for all $\nu, \eta \in X$ we can always determine the multiplicity of a given simple $\bar u_{r} \bar B_K$-module in a baby Verma module by passing to a case where the simple module has antidominant highest weight.  Hence  equation (\refeq{last_p}) means that 
the characters $\{\Char \bar T^\infty_p(\lambda) \mid \lambda \in X^-\}$ determines the composition factors of all baby Verma modules. This is equivalent to the determination of all simple $\bar u_{r}\bar B_K$-modules. In turn these determine all irreducible characters in $\bar {\mathcal O}_p$, see Corollary \ref{irr in Oq} The reverse implication works as well.

The proof in the quantum case is analogous.
\end{proof}

\begin{rem} If $p \geq 2h-2$ (or whenever Donkin's conjecture holds)  we can in Corollary \ref{last} replace the $\infty$-tilting modules with the finite dimensional tilting modules with highest weights in $(p^r-1)\rho + X^+$, respectively $(\ell p^r - 1)\rho) + X^+$, see Theorems 7.7(2) and 7.9(2). This result should be compared to \cite{So} where it is proved (without any condition on $p$) that there exists $r>1$ such that the finite dimensional indecomposable  tilting modules with highest weights in $(p^r-1)\rho + X_p$ determine all (finite dimensional) simple modules.
\end{rem}

\vskip 1 cm
\end{document}